\documentclass{amsart}

\usepackage{stackrel}
\usepackage[utf8]{inputenc}
\usepackage{amssymb}
\usepackage{amsthm}
\usepackage{tikz}
\usepackage{mathrsfs}
\usepackage{epsf}
\usepackage{bookmark,hyperref}
\hypersetup{
     colorlinks=true,
     linkcolor=blue,
     filecolor=blue,
     citecolor = blue,
     urlcolor=cyan,
     }
\usepackage{pstricks}

\usepackage{amsmath,amssymb,amscd}
\usepackage{amsthm}
\usepackage{subfigure}
\usepackage{verbatim}
\usepackage{fancybox}
\usepackage{subfigure}
\usepackage[all]{xy}

\newcommand{\N}{{\mathbb N}}

\newcommand{\Z}{{\mathbb Z}}

\newcommand{\R}{{\mathbb R}}

\newcommand{\I}{{\mathcal{O}}}

\newcommand{\dem}{{\em Proof: \;}}
\newcommand{\fdem}{\hfill $\square$}

\newtheorem{teo}{Theorem}[section]
\newtheorem{lema}[teo]{Lemma}
\newtheorem{cor}[teo]{Corollary}

\newtheorem{prop}[teo]{Proposition}
\newtheorem*{teosn}{Theorem}

\newtheorem{defi}{Definition}[section]

\makeatletter
\@namedef{subjclassname@2020}{%
  \textup{2020} Mathematics Subject Classification}
\makeatother

\emergencystretch=\maxdimen
\begin{document}

\title[On the Lagrange\,-\,Dirichlet converse in dimension three]{On the Lagrange\,-\,Dirichlet converse in dimension three}

\author{J. M. Burgos}
\address{Departamento de Matem\'aticas, CINVESTAV-\,CONACYT, Av. Instituto Polit\'ecnico Nacional 2508, Col. San Pedro Zacatenco, 07360 Ciudad de M\'exico, M\'exico.}
\email{burgos@math.cinvestav.mx}
\thanks{The first author is a research fellow at \textit{Consejo Nacional de Ciencia y Tecnolog\'ia}.}

\author{M. Paternain}
\address{Centro de Matem\'atica, Facultad de Ciencias, Universidad de la Rep\'ublica, Igu\'a 4225, 11400, Montevideo, Uruguay.}
\email{miguel@cmat.edu.uy}
\thanks{The second author was supported by the FCE\,-\,ANII\,-\,135352 grant.}

\begin{abstract}
Consider a mechanical system with a real analytic potential. We prove that in dimension three, there is an open and dense subset of the set of non strict local minimums of the potential such that every one of its points is a Lyapunov unstable equilibrium point.
\end{abstract}




\maketitle


\tableofcontents

\newpage
\section{Introduction}

This paper is about instability in mechanical systems. We start by recalling the standard concepts and notation. Let $M$ be a smooth manifold of dimension $n$ and let $L:TM\rightarrow\R$ be a Lagrangian defined on the tangent bundle $TM$ of the form
\begin{equation}\label{Lagrangian0}
L(x,v)= Q(x,v)-U(x),\qquad (x,v)\in TM.
\end{equation}

Here, $\left(x\mapsto Q_x\right)$ is a smooth section of positive definite quadratic forms, the \textit{kinetic term}, where we have defined $Q_x(v)=Q(x,v)$ and the potential $U$ is a smooth function on $M$. The resulting motion equations from a variational argument are the \textit{Euler-Lagrange equations}
$$\dfrac{d}{dt}\dfrac{\partial L}{\partial v_i}\left(x(t),\,\dot{x}(t)\right)\,-\, \dfrac{\partial L}{\partial x_i}\left(x(t),\,\dot{x}(t)\right)\,=\,0,\qquad
i=1,\ldots, n.$$

The Euler-Lagrange equations govern the dynamics of a \textit{mechanical system} in the \textit{phase space} $TM$. These equations show that critical points of the potential in the \textit{configuration space} $M$ correspond to equilibrium points of the system in the phase space. Specifically, if $p$ in $M$ is a critical point of the potential, then $(p,{\bf 0})$ in $TM$ is an equilibrium point of the dynamics.

An equilibrium point $(p,{\bf 0})$ in the phase space is \textit{Lyapunov stable} if for every neighbourhood $W$ of the point in the phase space, there is a smaller neighbourhood $W'\subset W$ of the point in the phase space such that every orbit starting at $W'$ do not leave $W$ in future time. An equilibrium is \textit{Lyapunov unstable} if it is not Lyapunov stable.

We will abuse of notation by saying that a point $p$ in the configuration space is an equilibrium point whenever the point $(p,{\bf 0})$ in the phase space is so and $p$ will inherit the properties of $(p,{\bf 0})$ as well. For example, we will say that $p$ is a Lyapunov stable equilibrium point if $(p,{\bf 0})$ is so.

The following theorem stated by J. L. Lagrange in \cite{Lagrange} and proved by L. G. Dirichlet in \cite{Dirichlet}, gives a sufficient condition on the critical point of the potential for the corresponding equilibrium point to be Lyapunov stable.
\begin{teosn}[J. L. Lagrange, L. G. Dirichlet]
Every strict local minimum of the potential is a Lyapunov stable equilibrium point.
\end{teosn}

In (\cite{Lyapunov}, Section 25), A. M. Lyapunov asks for the first time whether the converse of the previous theorem holds\footnote{The \textit{force-function} is the potential function, the opposite of the potential energy $U$.}:
\begin{quote}
\textit{``But, in establishing that this condition is sufficient, the theorem in question does
not allow any conclusion about the necessity of the same condition.
That is why the question arises: will the position of equilibrium be unstable if
the force\,-\,function is not maximum?"}
\end{quote}

The Lagrange\,-\,Dirichlet converse is false in the class of smooth potentials. The first example of this phenomenon was given by P. Painlev\'e in 1904 and the following is the A. Wintner's version \cite{Wi} of it: Consider the one degree of freedom Newtonian system with the smooth potential
$$U(x)= \exp(-x^{-2})\cos(x^{-1}),\,\ x\neq 0,\quad U(0)=0.$$
The origin is a critical point and it is not a local minimum. For every neighborhood of the origin, there is an interval centered at it and contained in the neighborhood such that the potential is maximum and strictly greater than zero on the interval's boundary. Therefore, for every neighborhood of the origin, every motion with small enough energy is trapped in an interval contained in the neighborhood hence the origin is Lyapunov stable and a counterexample of the Lagrange\,-\,Dirichlet converse.

A more striking example was proposed by M. Laloy in \cite{Laloy1}. Consider the two degrees of freedom Newtonian system with the smooth potential
$$U(x,y)= \exp(-x^{-2})\cos(x^{-1})- \exp(-y^{-2})\left(\cos(y^{-1}) + y^{2}\right),\,\ x,\,y\neq 0$$
and defined by its unique smooth extension on the coordinate axes. The origin is a critical point and it is not a local minimum. In contrast with the previous example, now there are no trapping zones since the set $U^{-1}\left((-\infty, 0]\right)$ contains the two diagonals $x=\pm y$ and these are escape routes to infinity where a priori any motion starting near the origin could take. However, the projections of the motion on the respective coordinates are decoupled trapped motions as in the first example hence the origin is again Lyapunov stable. This constitutes another counterexample of the Lagrange\,-\,Dirichlet converse, now without trapping zones.

The stability of equilibrium points in the class of smooth mechanical systems is subtler than one naively may think as the \textit{Bertotti-Bolotin phenomenon} shows (\cite{BB}, Theorem 1):

\begin{teosn}[M. L. Bertotti, S. V. Bolotin]
For two degrees of freedom, there exist a smooth potential $U$ and two smooth kinetic terms $Q$ and $\tilde{Q}$ such that the equilibrium ${\bf 0}$ is stable for the system $L=Q-U$ and unstable for the system $\tilde{L}=\tilde{Q}-U$.
\end{teosn}

In \cite{GT2}, M. V. P. Garcia and F. A. Tal gave an explicit example of the Bertotti-Bolotin phenomenon. In their example, the kinetic terms are independent of the configuration space coordinate. They also study the influence of the kinetic term in the real analytic class of mechanical systems and this influence is further studied by them and R. B. Bortolatto in \cite{BGT}. In the real analytic class they show the very interesting fact that although the instability nature of the equilibrium point is preserved for distinct real analytic kinetic terms, the dimension of the stable manifold containing the asymptotic orbits to the point may change.

Another result that clearly shows the subtlety of the stability problem in the class of smooth mechanical systems is \textit{Ureña's example} (\cite{Urena}, Theorem 1.1). Before stating this result, let us explain its importance.

It is well known since the work of P. Hagedorn in \cite{Ha} that a local strict maximum of a smooth potential is Lyapunov unstable for any smooth kinetic term. This result was later improved by S. D. Taliaferro in \cite{Taliaferro} for local maximums and M. Sofer in \cite{Sofer} for local maximums and generalized systems whose natural framework is Finsler geometry instead of Riemannian geometry. See also the variational approach by P. Hagedorn and J. Mawhin in \cite{Hagedorn2}.

However, it was expected that a strict maximum of a smooth potential were unstable also with respect to a stronger concept than Lyapunov instability and this is the concept of \textit{Siegel\,-\,Moser instability}: an equilibrium point in phase space is \textit{unstable in the sense of Siegel and Moser} if there is a neighbourhood of the point such that every other globally defined orbit has a point outside the neighbourhood or equivalently, if the equilibrium point is the maximal invariant subset in the neighbourhood. Note that a non isolated equilibrium point is never unstable in the sense of Siegel and Moser.

The previous naive thought about the Siegel\,-\,Moser instability of a strict maximum of a smooth potential turned out to be false as the following theorem shows:
\begin{teosn}[A. J. Ureña]
There exists a $C^\infty$ function $U:\R^2\rightarrow \R$ satisfying:
\begin{enumerate}
\item $U({\bf 0})=\max_{\R^2}\,U$.
\item $\nabla U(q)\neq {\bf 0}$ for any $q\neq{\bf 0}$.
\item There exists a sequence $T_n>0$ of positive numbers and a sequence
$$q_n:\R/\,T_n\,\Z\rightarrow\R^2$$
of nontrivial periodic solutions of $\ddot{q}=-\nabla U(q)$ such that
$$\max_{t\in\R}\,\left(\,\Vert q_n(t)\Vert+\Vert\dot{q}_n(t)\Vert\,\right)\to 0\quad\mbox{as}\quad n\to +\infty.$$
\end{enumerate}
\end{teosn}
In Ureña's example, the periods $T_n$ diverge and the closed orbits $q_n$ do not encircle the origin. In particular, this theorem states the existence of a smooth potential with a strict maximum which is not Siegel\,-\,Moser unstable.


Although the previous example is not directly related to the Lagrange\,-\,Dirichlet converse, it shows the unexpected behaviour that a mechanical system can have in the smooth class.


In view of the previous counterexamples and the Bertotti-Bolotin phenomenon, it is natural to ask for the Lagrange\,-\,Dirichlet converse in some restricted class of potentials, for example the class of real analytic potentials. Implicitly, this is the class wherein the original Lyapunov's question was formulated since he worked exclusively with real analytic functions all over his treatise. This problem will be called the \textit{analytic Lagrange\,-\,Dirichlet converse}.

For two degrees of freedom, this problem was completely solved. Indeed, V. P. Palamodov in \cite{Pa3} and S. D. Taliaferro in \cite{Ta} showed independently the instability of a non local minimum critical point of a real analytic potential in dimension two for a Newtonian kinetic term. Based on the previous Palamodov's work, this result was generalized by V. V. Kozlov in \cite{Ko} for an arbitrary kinetic term. 
The remaining case of a non strict local minimum was treated in \cite{Laloy} by M. Laloy and K. Peiffer who showed the instability of such a point for a real analytic potential. This completes the proof of the analytic Lagrange\,-\,Dirichlet converse in dimension two.

We find the following related weaker problem posed by V. I. Arnold in the \textit{Arnold's Problems book} \cite{Arnold}:
\begin{quote}
\textit{``}\textbf{1971-4.} \textit{Prove the instability of the equilibrium $\textbf{0}$ of the analytic system $\ddot{x}=-\partial U/\partial x$ in the case where the isolated critical point $\textbf{0}$ of the potential $U$ is not a minimum."}
\end{quote}

In \cite{Br}, M. Brunella gave an alternative and very elegant proof of the Arnold's problem for two degrees of freedom. We will get back to comment about this work in a moment.

Starting from A. M. Lyapunov \cite{Lyapunov} and continuing with \cite{Cetaev}, \cite{Ha}, \cite{Hagedorn2}, \cite{Ref2}, \cite{Ref3}, \cite{Ref4}, \cite{Ref5}, \cite{Ref6}, \cite{Sofer}, \cite{Taliaferro}, \cite{Ref7}, many partial results have been given towards the solution of the Arnold's problem and their common feature is that the Lyapunov instability criteria involves the lack of a local minimum at the origin of the first nonzero jet of the potential. In particular, these results are not sufficient to solve the Arnold's problem in the general case. As an example, consider the following potential in $\R^3$,
$$U(x,y,z)\,=\,x^4+y^4-z^6.$$
Here, the origin is an isolated critical point of the potential that is not a local minimum of it. However, the origin is a minimum of the first non zero jet,
$$(j^{(4)}\,U)\,(x,y,z)\,=\,x^4+y^4$$
hence none of the mentioned results apply for the potential $U$.

In \cite{Palamodov}, V. P. Palamodov stated a \textit{total instability} result for non minimum critical points of the potential whose proof is incomplete. This total instability conjecture would imply the Arnold's problem and the impossibility of Ureña's example in the class of real analytic potentials as well. All of these remain as open problems.


The only general result on the analytic Lagrange\,-\,Dirichlet converse, in the sense that there is no other hypothesis than the real analyticity of the potential and the dimension of the configuration space, is the one previously described in dimension two. We prove for the first time the following general result on the analytic Lagrange\,-\,Dirichlet converse in dimension three.

\begin{teo}\label{main}
Consider a mechanical system on a three dimensional real analytic manifold with a real analytic potential. Then, there is an open and dense subset of the set of non strict local minimums of the potential such that every one of its points is a Lyapunov unstable equilibrium point.
\end{teo}

For the proof, we develop a new instability criterion for non strict local minimums of a smooth potential in arbitrary dimension. This is proved in section \ref{instability_criterion} and generalizes previous work by the authors in \cite{BMP} and \cite{BP}. Besides the work by M. Laloy and K. Peiffer in \cite{Laloy} in dimension two, as far as our knowledge goes there is no other instability criterion for non strict local minimums of the potential. 
For this reason, we consider that the mentioned new instability criterion has interest on its own, specially in mechanics.

Because of the fact that the instability criterion does not involve the kinetic term, none of the unstable equilibrium points described in Theorem \ref{main} exhibit the Bertotti\,-\,Bolotin phenomenon, that is to say they are Lyapunov unstable in the Lagrangian dynamics of \eqref{Lagrangian0} with an arbitrary kinetic term.

The instability criterion requires a new concept in its hypothesis. In section \ref{Weakly_log_section} we introduce the notion \textit{weakly logarithmic vector fields} and motivate the necessity of the definition through examples. We show that the direct analog of the known logarithmic vector fields in complex analytic geometry is too rigid and restrictive in real analytic geometry. We justify the new concept by proving a local extension theorem for codimension two real analytic submanifolds that is unimaginable in complex analytic geometry or even in real algebraic geometry using the direct analog of logarithmic fields.
Again, for this reasons we consider that the new concept of weakly logarithmic vector fields introduced here as well as the results concerning it, have interest on its own, specially in real algebraic and analytic geometry.


As it was mentioned before, using desingularization results for singular vector fields in three dimensional space developed by F. Cano in \cite{Cano1} and \cite{Cano2}, in \cite{Br} M. Brunella gives sufficient conditions for a real analytic vector field with an isolated singularity in three dimensional space to be unstable at the singular point. Applying his result to an appropriate energy level of a hamiltonian vector field in four dimensional space, he solves the Arnold's conjecture for two degrees of freedom. This way, Brunella's result applied to hamiltonian vector fields is about instability in two dimensional configuration space while ours concerns instability in three dimensional configuration space.

We strongly use desingularization techniques in our proof. Specifically, we use the locally finite functorial principalization theorem for real and complex analytic spaces proved by J. W\l{}odarczyk in \cite{Wlodarczyk} building upon previous work of H. Hironaka \cite{Hironaka}, O. Villamayor  \cite{Villamayor1}, \cite{Villamayor2}, E. Bierstone and P. D. Milman \cite{Bierstone_Milman}. Along the well known blowing up technique of vector fields, we develop a blowing down result specially suited for our purposes for vector fields in arbitrary dimensions that we also believe that it has interest in its own, specially in the desingularization theory of vector fields.

Although we do not use any of his results, our algorithm for the construction of a weakly logarithm field is inspired in the idea of V. P. Palamodov in \cite{Palamodov}. Essentially, his idea consists in adding corrections to a given vector field along each step of a principalization of the potential and then coming back to the zeroth step with the corrected vector field satisfying the desired logarithmic properties.

We do not know whether there is a stable non strict local minimum point of a real analytic potential in the Lagrangian dynamics of \eqref{Lagrangian0}. However, as an immediate corollary of the mentioned instability criterion, such a point cannot admit a weakly logarithmic vector field non null at the point. On the other hand, there are unstable non strict local minimum points that do not admit non null weakly vector fields, that is to say the criterion is sufficient but not necessary. The moral is that non strict local minimum points not admitting non null weakly logarithmic fields are natural candidates for counterexamples of the analytic Lagrange\,-\,Dirichlet problem.

\section{Weakly logarithmic vector fields}\label{Weakly_log_section}

\subsection{Preliminaries on real analytic geometry}\label{preliminaries_real_analytic}

In this preliminary section, we will follow closely the notes of D. B. Massey and L. D. Tr\'ang \cite{Trang_Massey}, specially sections $5$ and $6$. For details we refer to these notes and references therein.
A real analytic manifold is a manifold with an atlas whose transition functions are real analytic. A function between real analytic manifolds is real analytic if its expression with respect to any pair of coordinate neighbourhoods of the respective atlases is so. The set of critical points of a function $f$ will be denoted by $\Sigma f$. Unless explicitly specified, in this subsection $M$ will denote a real analytic manifold \footnote{In \cite{Trang_Massey} it is assumed that the manifold $M$ is also connected. However, a careful inspection shows that we only need that hypothesis in the analytic continuation principle, Theorem \ref{analytic_continuation}.}.

Two real analytic functions $f$ and $g$ define the same \textit{germ} at a point if there is a neighbourhood $A$ of the point whereat their restrcictions are equal, $f|_A=g|_A$. Two subsets $S$ and $T$ of $M$ define the same germ at a point if there is a neighbourhood $A$ of the point such that $S\cap A=T\cap A$. We will denote the germ of the subset $S$ at $p$ simply by $S_p$.

The ring of real valued real analytic functions on $M$ will be denoted by $\I_M$ and the ring of germs at the point $p$ in $M$ will be denoted by $\I_{M,\,p}$.

\begin{teo}[Theorem 5.2, \cite{Trang_Massey}]
The ring $\I_{M,p}$ is a Noetherian unique factorization domain (UFD). The ring $\I_M$ is an integral domain, but need not be Noetherian or a UFD.
\end{teo}

The following theorem is known as the \textit{analytic continuation principle}.

\begin{teo}[Theorem 5.1, \cite{Trang_Massey}]\label{analytic_continuation}
Suppose that $M$ is connected. Consider a point $p$ in $M$ and a real analytic function $f$ in $\I_M$. If the germ of $f$ at $p$ is zero, then the function $f$ is identically zero.
\end{teo}

\begin{defi}
For any set $S$ of real analytic functions, define its zero locus by
$$\mathcal{V}(S)=\{\,x\in M\ |\ f(x)=0,\,\forall\,f\in S\,\}.$$
\end{defi}

\begin{defi}
A subset $X\subset M$ is analytic if it is closed and locally it is the zero locus of some finite set of real analytic functions, that is to say, for every point $x$ in $X$ there is a neighbourhood $W$ of $x$ in $M$ and there are real analytic functions $f_1,\,f_2,\ldots,\, f_j$ in $\I_W$ such that
$$W\cap X=\mathcal{V}(f_1,\,f_2,\ldots,\,f_j).$$

A subset $E\subset M$ is locally analytic if, for all $p$ in $E$, there exists an open neighborhood $W$
of $p$ in $M$ such that $W\cap E$ is an analytic subset of $W$.
\end{defi}

There are variants of the previous definition in the literature, see for instance \cite{Trang_Massey}, page $12$. There are analytic subsets not defined as the zero locus of globally defined analytic functions in $M$, see the striking example given by H. Cartan in (\cite{Cartan}, section 11). Note that locally analytic subsets need not to be closed.

\begin{teo}[Theorem 5.7, \cite{Trang_Massey}]
The intersection of analytic subsets is analytic.
\end{teo}

\begin{cor}
The analytic subsets in $M$ are the closed sets of a topology in $M$. This is the analytic Zariski topology.
\end{cor}

Without further qualification, the terms \textit{open} and \textit{closed} will refer to the classical topology. In particular, in the next section \textbf{the principalization will be over classical open sets}.

\begin{defi}
Consider an analytic subset $X\subset M$. We will say that $X$ is irreducible if it cannot be written as the union of proper analytic subsets.
We will say that the germ $X_p$ is irreducible if it cannot be written as the union of two proper germs of analytic spaces at $p$. We will say that $X$ is irreducible at $p$ if the germ $X_p$ is irreducible.
\end{defi}

Note that every irreducible analytic subset is connected but the converse is false in general.

As an example consider $X=\mathcal{V}(y^2-x^3-x^2)\subset \R^2$. This is an irreducible analytic subset which is not irreducible at the origin. In effect, in the open unit ball $B$ we have the factorization
$$y^2-x^3-x^2\,=\,(y\,+\,x\,\sqrt{x+1})\,(y\,-\,x\,\sqrt{x+1})$$
where both factors on right hand side are analytic hence $X\cap B$ is the union of proper analytic spaces,
$$X\cap B\,=\,\mathcal{V}(y\,+\,x\,\sqrt{x+1})\,\cup\,\mathcal{V}(y\,-\,x\,\sqrt{x+1}).$$

\begin{defi}
Consider an analytic subset $X\subset M$ and a point $p$ in $X$. We will say that $p$ is a $d$-dimensional smooth point of $X$ if there is a neighbourhood $W$ of $p$ in $M$ such that $W\cap X$ is a $d$-dimensional analytic submanifold of $W$.
\end{defi}

We will denote by $\mathring{X}$ the set of smooth points of $X$ and by $\mathring{X}^{(d)}$ the set of $d$-dimensional smooth points of $X$. A point in $X$ is \textit{singular} if it is not smooth and we will denote the set of singular points of $X$ by $\Sigma X$. It is clear that we have a decomposition
$$\mathring{X}\,=\,\mathring{X}^{(0)}\sqcup\mathring{X}^{(1)}\sqcup\ldots\sqcup \mathring{X}^{(n)}.$$

\begin{teo}[Theorem 5.14, \cite{Trang_Massey}]\label{smooth_points}
Consider an analytic subset $X\subset M$. Then,
\begin{enumerate}
\item\label{open_submanifold} For every $d$, the set $\mathring{X}^{(d)}$ is an analytic $d$-dimensional submanifold in $M$ and an open set in $X$.
\item\label{open_dense} $\mathring{X}$ is an open and dense subset in $X$.
\end{enumerate}
\end{teo}

As an immediate corollary, we have that the set of singular points $\Sigma X$ is a closed nowhere dense subset in $X$. In general, $\Sigma X$ will not be analytic even if $X$ is so.

\begin{defi}
A regular component of an analytic subset $X\subset M$ is a connected component of $\mathring{X}$.
\end{defi}

\begin{defi}
The dimension of an analytic set $X\subset M$ is the largest $d$ such that $\mathring{X}^{(d)}$ is non empty. The dimension of $X$ at a point $p$ in $X$, is the largest $d$ such that $p$ is in the closure of $\mathring{X}^{(d)}$. We say that $X$ is pure dimensional if and only if the dimension of $X$ at each one of its points is independent of the point. We make the analogous definitions for local analytic sets.
\end{defi}

It follows immediately from the definitions that a connected component of an analytic subset $X\subset M$ belongs to some $\mathring{X}^{(d)}$ hence it is a relatively open pure dimensional locally analytic subset in $X$. In particular, every regular component of an analytic subset has a well defined dimension.

\begin{teo}[Theorem 5.24, \cite{Trang_Massey}]\label{fully_contained}
Suppose that $X$ is an analytic subset of $M$, and $E$ is a connected, pure dimensional, locally analytic subset of M such that, for all $x$ in $E$, $E_x$ is irreducible. Finally, suppose that $p$ in $E$ is such that $E_p\subset X_p$. Then, $E \subset X$.
\end{teo}


The following result is due to J. Souček and V. Souček in \cite{Soucek}. See \cite{Trang_Massey} for a proof using the curve selection Lemma.

\begin{teo}[Theorem 6.7, \cite{Trang_Massey}]\label{isolated_critical_values_2}
Consider a real analytic function $f$ on $M$. Then, for every point $p$ in $M$ there is a neighbourhood $W$ of $p$ in $M$ such that
$$W\cap \Sigma f\,\subset\,f^{-1}(f(p)).$$
In particular, for every critical point $p$ of $f$ there is a neighbourhood $W$ of $p$ in $M$ such that $f(p)$ is the only critical value of $f|_W$, that is to say, locally the critical values are isolated.
\end{teo}

\begin{cor}
Consider a real analytic function $f$ on $M$. Then, for every precompact neighbourhood $A$ in $M$, the set of critical values of $f|_A$ is finite.
\end{cor}

We end this section with some basic differential geometric results. The following is the Brown-Sard theorem.

\begin{teo}[Sections 2 and 3, \cite{Milnor}]
Consider smooth manifolds $M$ and $N$ and a smooth function $f:M\rightarrow N$. Then, the set of critical values $f(\Sigma f)\subset N$ has zero Lebesgue measure. In particular, it is nowhere dense in $N$.
\end{teo}

The last sentence of the previous Theorem is the original statement due to A. B. Brown while the the first is the improvement due to A. Sard following earlier work by A. P. Morse.

\begin{cor}\label{Sard}
Consider smooth manifolds $M$ and $N$ and a proper smooth function $f:M\rightarrow N$. Then, the set of critical values $f(\Sigma f)\subset N$ is a closed  zero Lebesgue measure subset.
\end{cor}
\begin{proof}
Because $N$ is locally compact and $f$ is proper, we have that $f$ is a closed map hence $f(\Sigma f)\subset N$ is a closed subset since $\Sigma f\subset M$ is so. By the previous Brown-Sard's Theorem, the result follows.
\end{proof}

We introduce the following definition.

\begin{defi}\label{strong_reg_seq}
A sequence $y_1,\,y_2,\,\ldots,\,y_m$ of real analytic functions on some neighbourhood $W$ in $M$ is a strong regular sequence if the map
$$(y_1,\,y_2,\,\ldots,\,y_m)\,:W\rightarrow \R^m$$
is, in differential geometric terms, regular, that is to say the differential of the map at every point in $W$ is surjective.
\end{defi}

The previous definition alludes to the fact that every strong regular sequence is a regular sequence in the the sense of commutative algebra. The zero locus of a strong regular sequence on $W$ is a real analytic submanifold in $W$. Locally, the converse also holds.

Explicitly, $y_1,\,y_2,\,\ldots,\,y_m$ is a strong regular sequence if at every point $p$ in $W$, the differentials
$$d_p y_1,\ d_p y_2,\,\ldots,\ d_p y_m$$
constitute a linearly independent set in $T_p W^*$. In particular, by the implicit function Theorem, every strong regular sequence $y_1,\,y_2,\,\ldots,\,y_m$ on $W$ is part of a strong regular sequence $y_1,\,y_2,\,\ldots,\,y_n$ on a possibly smaller neighbourhood $W'\subset W$ such that the resulting map
$$(y_1,\,y_2,\,\ldots,\,y_n)\,:W'\rightarrow \R^n$$
is a coordinate neighbourhood.

\subsection{Definition and statement of the theorems}\label{weakly_log_subsection}

\begin{defi}[Logarithmic field]
Consider a real analytic function $U$ in a real analytic manifold $M$. A vector field $V$ defined on some open subset $A\subset M$ is logarithmic for $U$ in $A$ if there is a real analytic function $P$ defined on $A$ such that $V(U)= P\,U$ on $A$.
\end{defi}

The previous definition has the following direct analog in the class of smooth functions.

\begin{defi}[Smoothly logarithmic field]
Consider a smooth function $U$ in a smooth manifold $M$. A vector field $V$ defined on some open subset $A\subset M$ is smoothly logarithmic for $U$ in $A$ if there is a smooth function $P$ defined on $A$ such that $V(U)= P\,U$ on $A$.
\end{defi}

A classical example of a (smoothly) logarithmic vector field is the one provided by Euler's Theorem on quasi homogeneous polynomials (smooth functions) which we shall describe. Consider a quasi homogeneous polynomial (smooth function) $f$ on $\R^n$ such that
$$f(t^{\lambda_1}\,x_1,\ldots,\,t^{\lambda_n}\,x_n)\,=\,t^\lambda\,f(x_1,\ldots,\,x_n),\qquad t>0$$
where the $\lambda$'s are constants. By Euler's Theorem we have
$$V(f)\,=\,\lambda\,f,\qquad\quad V\,=\,\sum_{i=1}^n\,\lambda_i\,x_i\,\partial_i$$
hence $V$ is a (smoothly) logarithmic vector field for $f$.

In the rest of this subsection, to motivate the necessity of a new notion of logarithmic fields, specifically the notion of a \textit{weakly logarithmic vector field}, we will assume that $U$ is a real analytic function on a real analytic manifold $M$.

Note that a vector field logarithmic for $U$ in some open subset intersecting some regular component of the zero locus of $U$ is necessarily tangent to this component. Here and in the rest of the paper, \textbf{the tangency property of a vector includes the possibility of the vector being zero.}

Note also that every real analytic vector field defined on some open subset $A\subset M$ is always logarithmic for $U$ in $A-\mathcal{V}(U)$. The troublesome set for a vector field to be logarithmic for $U$ is always the zero locus of $U$.

Now, in contrast to the complex case, a real analytic vector field tangent to some regular component of the zero locus of $U$ may not be logarithmic on the component even if the component has codimension one. As an example, consider the potential in $\R^2$.
\begin{equation}\label{example1}
U(x,y)=x^2(x^2+y^2).
\end{equation}
The zero locus is the $y$-axis and it is clearly a smooth irreducible real analytic submanifold. However, there is no real analytic logarithmic for $U$ vector field non null at the origin. In effect, without loss of generality and normalizing the field on some neighbourhood of the origin if necessary, we may suppose that the vector field is
$$V=\partial_y+ x\,f\,\partial_x$$
on some neighbourhood of the origin. Here $f$ is a real analytic function. Then,
$$P(x,y)=(V(U)/U)(x,y)=2\,\frac{y}{x^2+y^2}+2\,f\,\left(1+\frac{x^2}{x^2+y^2}\right),\quad x\neq 0.$$
The second term is bounded over compact sets while the first term diverges to infinity as the point goes to the origin along the diagonal:
$$\lim_{x\to 0}\,P(x,x)=\infty.$$
In particular, $V$ is not logarithmic at the origin.

A more striking example is the following adaptation of the Whitney example. Consider the potential
\begin{equation}\label{example2}
U(x,y,z)=x^2\,y^2\,(x-y)^2\,(x-z\,y)^2
\end{equation}
and suppose that there is a real analytic logarithmic vector field on some neighbourhood of the point $(0,0,z)$ with $0<z<1$ and non null at the point. The vector field must be a non zero multiple of $\partial_z$ at the point. Without loss of generality, normalizing and taking a smaller neighbourhood if necessary, by continuity we may suppose that the $z$-component of the vector field equals to $\partial_z$ on the neighbourhood of the point.
Then, the vector field generates a flow whose differential restricts to the orthonormal planes with the $z$-axis. In particular, it is a linear isomorphism of the plane fixing the axes and the diagonal but mapping the straight lines $x=z\,y$ with different slope and this is a contradiction. Therefore, none of the points $(0,0,z)$ with $0<z<1$ admit a real analytic logarithmic vector field on some neighbourhood of the point and non null at it.

In section \ref{Codim1} we prove the following theorem:

\begin{teo}\label{Theorem_codim1}
Consider a real analytic function $U$ in a real analytic manifold $M$ and a codimension one regular component $X$ of its zero locus $\mathcal{V}(U)$. Then, there is an open subset $A\subset M$ and a closed nowhere dense zero Lebesgue measure subset $S\subset X$ such that
\begin{enumerate}
\item $A\cap \mathcal{V}(U)=X$.
\item Every real analytic (smooth) vector field defined on some open subset $B\subset A$ and tangent to $X$ at $B\cap X$, possibly empty, is (smoothly) logarithmic for $U$ in $B-S$.
\end{enumerate}
\end{teo}

The previous Theorem is no longer valid in codimension two. As an example, consider the potential in $\R^3$
\begin{equation}\label{example3}
U(x,y,z)=x^2+y^2.
\end{equation}
Its zero locus is the $z$-axis. However the real analytic vector field
\begin{equation}\label{example3_vectorfield}
V=x\,\partial_x+2\,y\,\partial_y+\partial_z
\end{equation}
is tangent to the zero locus and it not logarithmic at any point of it. In effect,
$$P(x,y)=(V(U)/U)(x,y)=2\,\left(1+\frac{y^2}{x^2+y^2}\right)=2\,(1+\sin^2(\theta)),\quad r\neq 0$$
where we have expressed the result in the cylindrical coordinates along the $z$-axis. It is clear that $P$ does not have a continuous extension at $r=0$ since the limit depends on the $\theta$ angle. In particular, $V$ is not logarithmic at any point of the zero locus.

The previous example motivates the following definition:

\begin{defi}[Weakly logarithmic field]
Consider a smooth function $U$ in a smooth manifold $M$. A vector field $V$ defined on some open subset $A\subset M$ is weakly logarithmic for $U$ in $A$ if there is a function $P$ defined on $A$ and bounded on compact subsets such that $V(U)= P\,U$ on $A$.
\end{defi}

Every (smoothly) logarithmic vector field is weakly logarithmic and the vector field \eqref{example3_vectorfield} is an example of a weakly logarithmic field that is not (smoothly) logarithmic for the potential \eqref{example3}.

Looking forward for its dynamical applications, concretely the application of our instability criterion proved in section \ref{instability_criterion}, we are mainly interested in weakly logarithmic vector fields with certain regularity and we ask for this regularity to be locally Lipschitz continuous at least.

The existence of locally Lipschitz continuous weakly logarithmic vector fields is a subtle matter. Indeed, the same conclusions of the previous examples \eqref{example1} and \eqref{example2} hold for these fields as well with an almost verbatim argument. Moreover, the example \eqref{example1} has the following codimension two version
$$U(x,y,z)=(x^2+y^2+z^2)(x^2+y^2).$$
An almost verbatim argument as in the example \eqref{example1} shows the non existence of locally Lipschitz continuous weakly logarithmic vector fields for $U$ non null at the origin.


In section \ref{Codim2} we prove the following local extension theorem:

\begin{teo}\label{Theorem_codim2}
Consider a real analytic function $U$ in a real analytic manifold $M$ and a codimension two regular component $X$ of its zero locus $\mathcal{V}(U)$. Consider also a smooth vector field $F$ defined on $X$ and tangent to it. Then, there is a closed nowhere dense zero Lebesgue measure subset $S\subset X$ such that for every point $p$ in $X-S$ there is a neighbourhood $W_p$ of $p$ in $M$ whereat $F$ admits a locally Lipschitz continuous weakly logarithmic for $U$ extension.
\end{teo}


\subsection{Blowup and Ehresmann connection}\label{connection}
For this section we recommend the recent paper of M. Spivakovsky \cite{Mark} and references therein, specially the introduction of section 1.3 and subsections 1.3.1 and 1.3.2. For the concept of Ehresmann connection we recommend the classic book of S. Kobayashi and K. Nomizu \cite{Kobayashi_Nomizu}, chapter 2, section 1.

Consider a real analytic manifold $M$ and a closed real analytic submanifold $C$ of codimension $m$ at least one. Without loss of generality, we may suppose that $C$ is affine and is given by the zero locus of a strong regular sequence $y_1,\ldots, y_m$ of real analytic functions on $M$:
$$C= \mathcal{V}(y_1, \ldots, y_m).$$

Consider the real projective space $\mathsf{P}^{m-1}(\R)$ and denote by $\theta_1,\ldots, \theta_m$ its homogeneous coordinates. Define the blowup $Bl_C(M)\subset M\times \mathsf{P}^{m-1}(\R)$ of $M$ along $C$ as the variety of zeroes of the ideal generated by the relations $y_i\,\theta_j- y_j\,\theta_i$ for every $i,j=1,\ldots, m$. Define the canonical projection $\pi$ as the projection on the first factor. We summarize the definition in the following commutative diagram
\begin{equation}\label{blowup_construction}
\xymatrix{
Bl_C(M)=\mathcal{V}(y_i\,\theta_j- y_j\,\theta_i)\ \ \ar[dr]_\pi \ar@{^{(}->}[r] & \ \ M\times \mathsf{P}^{m-1}(\R) \ar[d]^{pr_1} \\
& M}.
\end{equation}

The submanifold $C$ is the \textit{center} of the blowup and $\pi^{-1}(C)$ is the \textit{exceptional divisor}. The canonical projection induce an isomorphism outside the center
$$\pi: Bl_C(M)-\pi^{-1}(C)\overset{\cong}\longrightarrow M-C$$
while it is a locally trivial fibration over $C$ with fiber the projective space $\mathsf{P}^{m-1}(\R)$ and total space the exceptional divisor
$$\pi: \pi^{-1}(C)\longrightarrow C.$$
Note that the exceptional divisor is a real analytic submanifold.

The previous construction does not depend on the strong regular sequence defining $C$. In effect, consider another blowup $Bl_C(M)'$ constructed from another strong regular sequence $y_1',\ldots, y_m'$. Then there is a unique isomorphism $\iota:Bl_C(M)\rightarrow Bl_C(M)'$ such that $\iota\circ \pi'=\pi$. This allows to extend the construction for arbitrary possibly not affine submanifolds $C$.

Topologically, the blowup space can be constructed as follows: After removing a tubular neighbourhood of the center, the boundary of the resulting space is a sphere bundle whose projectivization gives the blowup space. From this construction as well as the definition, it is clear that \textbf{the blowup of a codimension one submanifold is isomorphic to the original space and its only purpose is to turn the center into an exceptional divisor.}

As an example, consider the blowup of an $n$-dimensional real analytic manifold $M$ whose center consist of a single point $p$. As a smooth manifold manifold, the blowup space is the connected sum of $M$ with the real projective space $\mathsf{P}^{n}(\R)$,
$$Bl_p\,(M)\,=\,M\,\,\#\,\,\mathsf{P}^{n}(\R).$$

Consider a real analytic variety $N\subset M$. We define the \textit{strict transform} $\pi^*(N)$ as the closure of $\pi^{-1}(N-C)$ in $Bl_C(M)$, that is to say
\begin{equation}\label{strict_transform_definition}
\pi^*(N)= \overline{\pi^{-1}(N-C)}\subset Bl_C(M).
\end{equation}

As an example, consider the cone $\mathcal{V}(y^2+z^2-x^2)\subset \R^3$. It has a singularity at the origin. It is easy to see from the topological construction that its strict transform in the blowup of $\R^3$ at the origin is a cylinder. This is a very basic example of \textit{desingularization}.

Indeed, the blowup of $\R^3$ at the origin is the variety in $\R^3\times \mathsf{P}^{2}(\R)$ given by the relations
\begin{equation}\label{Example_Relations}
x\,\theta_2-y\,\theta_1=0,\quad x\,\theta_3-z\,\theta_1=0, \quad y\,\theta_3-z\,\theta_2=0
\end{equation}
where $\theta_1, \theta_2, \theta_3$ are the homogeneous coordinates of $\mathsf{P}^{2}(\R)$. In the affine coordinate chart corresponding to $\theta_1\neq 0$, define the affine coordinates
$$\xi_2= \theta_2/\theta_1,\quad \xi_3= \theta_3/\theta_1.$$
Restricted to this coordinate chart, the exceptional divisor is the set of solutions of $x=0$ and because of the relations \eqref{Example_Relations}, we have
$$\pi(x,\xi_2, \xi_3)=(x,y,z),$$
$$y=x\,\xi_2,\quad z=x\,\xi_3.$$
Then, restricted to the coordinate chart corresponding to $\theta_1\neq 0$, the preimage of $\mathcal{V}(y^2+z^2-x^2)\subset \R^3$ by $\pi$ is
$$\mathcal{V}\left(x^2\,(\xi_2^2+\xi_3^2-1)\right)=\mathcal{V}(E)\cup\mathcal{V}(\xi_2^2+\xi_3^2-1)\subset Bl_{{\bf 0}}(\R^3)$$
where the first factor on the right hand side corresponds to the exceptional divisor and the second to the strict transform which clearly is a cylinder.

Suppose that the center $C$ is contractible. This is achievable by considering a small enough neighbourhood of $M$. In particular, the fibration $\pi:\pi^{-1}(C)\rightarrow C$ over the center is trivial hence
$$\pi^{-1}(C)\cong C\times \mathsf{P}^{m-1}(\R)\subset Bl_C(M)$$
where the isomorphism is the restriction of \eqref{blowup_construction} to the center $C$ which clearly depends on the chosen strong regular sequence. The trivial fibration has a flat Ehresmann connection defined by the horizontal distribution
$$\mathsf{Hor}=\left\lbrace\, C\times \{y\}\ |\ y\in \mathsf{P}^{m-1}(\R)\,\right\rbrace$$
while the vertical distribution is given by the set of fibers
$$\mathsf{Vert}=\left\lbrace\, \{x\}\times \mathsf{P}^{m-1}(\R)\ |\ x\in C\,\right\rbrace.$$

Although the blowup of a space is a universal object unique up to isomorphism, in contrast to the vertical distribution that is canonical, in general the horizontal distribution is not invariant under these isomorphisms hence, as explained before, this distribution depends on the chosen strong regular sequence $y_1, \ldots, y_l$. Equivalently, \textbf{the choice of a strong regular sequence is a choice of a flat Ehresmann connection on the exceptional divisor.}

\subsection{Pullback and Pushout}

In the whole section we assume that \textbf{the center is given by a strong regular sequence}. As it was explained in the previous section, the choice of a strong regular sequence is a choice of a flat Ehresmann connection on the exceptional divisor. For this reason, the following geometric constructions depend on the chosen strong regular sequence for the center.

\begin{lema}[Pullback]\label{Pullback}
Let $V$ be a smooth vector field in $M$ tangent to $C$. Then, there is a unique smooth vector field $V^{*}$ in $Bl_C(M)$ tangent to $\pi^{-1}(C)$ such that:
\begin{enumerate}
\item It satisfies $d_a \pi (V_a^{*})= V_{\pi(a)}$ for every $a$ in $Bl_C(M)$.
\item Restricted to $\pi^{-1}(C)$, it is the sum of the horizontal lifting of $V_C$ and a vertical vector field.
\end{enumerate}
\end{lema}

Because the pullback of a smooth vector field is unique, in general there will be no smooth pushout. However, under necessary tangential conditions, there exists a unique pushout in the class of locally Lipschitz continuous vector fields.

\begin{lema}[Pushout]\label{Pushout}
Let $V$ be a locally Lipschitz continuous vector field in $Bl_C(M)$ tangent to $\pi^{-1}(C)$ such that, restricted to $\pi^{-1}(C)$, it is the sum of the horizontal lifting of some vector field $W$ on $C$ and a vertical vector field. Then, there is a unique locally Lipschitz continuous vector field $V_*$ in $M$ tangent to $C$ such that:
\begin{enumerate}
\item It satisfies $d_a \pi (V_a)= V_{*,\pi(a)}$ for every $a$ in $Bl_C(M)$.
\item Restricted to $C$, it is equal to $W$.
\end{enumerate}
\end{lema}

Consider a point $p$ in $C$. Because $C$ is smooth, there is a neighborhood $W$ of $p$ in $M$ and a strong regular sequence $x_1,\ldots, x_k, y_1,\ldots y_l$ on $W$ such that
$$C\cap W= \mathcal{V}(y_1,\ldots, y_l)$$
and $x_1,\ldots, x_l$ are coordinates of $C\cap W$. The blowup $Bl_{C\cap W}(W)$ is covered by affine charts corresponding to $\theta_i\neq 0$. For the chart corresponding to $\theta_1\neq 0$ define the affine coordinates $\xi_2,\ldots, \xi_l$ by the quotient
$$\xi_i=\theta_i/\theta_1,\qquad i=2,\ldots,l.$$
Because of the relations $y_i\,\theta_j- y_j\,\theta_i=0$ on the blowup, the set of coordinates $x_1,\ldots, x_k, y_1,\xi_2,\ldots, \xi_l$ on the blowup is related to the set of local coordinates on $M$ by
$$\pi(x_1\ldots, x_k, y_1, \xi_2, \ldots, \xi_l)=(x_1\ldots, x_k, y_1, y_2, \ldots, y_l),$$
$$y_i= y_1\,\xi_i,\quad i=2,\ldots, l.$$
Note that, by definition of the coordinates $\xi_i$, these are defined on the region $y_1=0$. Actually, the region $y_1=0$ is the intersection of the exceptional divisor with the chart corresponding to $\theta_1\neq 0$ and $\xi_2, \ldots, \xi_l$ are the coordinates of this chart of the exceptional divisor. The coordinates on the other affine coordinate charts of the blowup are defined analogously.

With respect to these coordinates and restricted to the chart corresponding to $\theta_1\neq 0$, a vertical field in $\pi^{-1}(C)$ is spanned by the fields
$$\partial_{\xi_2},\ \partial_{\xi_3},\ldots,\ \partial_{\xi_l}$$
while a horizontal field is spanned by the fields
$$\partial_{x_1},\ \partial_{x_2},\ldots,\ \partial_{x_k}.$$
An analogous assertion holds for the other charts.

Because $\pi$ is an isomorphism outside the exceptional divisor, there are well defined pulback and pushout of vector fields on the corresponding regions. In this respect, for the proof of the previous Lemmas, we need the following calculation Lemma. The proof is a tedious but very elementary calculation that we do not include.

\begin{lema}\label{Calculation}
With respect to the coordinate chart corresponding to $\theta_1\neq 0$ and defining $\xi_1=1$ for notational convenience, in the region $y_1\neq 0$ we have:
\begin{enumerate}
\item $\left(y_j\,\partial_{y_i}\right)^{*}= \xi_j\,\partial_{\xi_i},\qquad i\neq 1.$
\item $\left(y_j\,\partial_{y_1}\right)^{*}= \xi_j\,\left(y_1\,\partial_{y_1} - \sum_{i=2}^{l}\,\xi_i\,\partial_{\xi_i}\right).$
\item $\left(\partial_{\xi_i}\right)_* = y_1\,\partial_{y_i},\qquad i\neq 1.$
\item $\left(y_1\,\partial_{y_1}\right)_*= \sum_{i=1}^{l}\,y_i\,\partial_{y_i}.$
\end{enumerate}
Analogous expressions hold for the other charts.
\end{lema}

It is interesting that even if the vector field is zero at the center, the pullback may be a non null vertical field at the exceptional divisor. As an example, consider the blowup of the origin in the plane $\pi:Bl_{\bf 0}(\R^2)\rightarrow\R^2$ and the vector field $V=y_1\,\partial_{y_2}$ on the plane with coordinates $y_1, y_2$. Note that the vector field $V$ vanishes on the origin. However, with respect to the coordinates $y_1,\xi_2$ on the region $y_1\neq 0$ of the blowup space described before, by the first item of Lemma \ref{Calculation} the pullback of $V$ by $\pi$ is $V^*= \partial_{\xi_2}$ which clearly smoothly extends to $y_1=0$. In particular, the pullback $V^*$ on $\pi^{-1}({\bf 0})$ is a nonzero vertical vector field.

\begin{proof}[Proof of Lemma \ref{Pullback}]
Because the vector field is smooth and due to its tangency hypothesis, restricted to the coordinate neighbourhood with coordinates $x_1,\ldots, x_l, y_1,\ldots y_k$ we have the expression:
$$V= \sum_{j,i=1}^{l}\,h_{ji}\,y_j\,\partial_{y_i} + \sum_{i=1}^{k}\,g_i\,\partial_{x_i}$$
for some smooth real valued functions $h_{ij}$ and $g_i$. On the complement of $E$, the exceptional divisor, there exists a pullback $V^{*}$ and by Lemma \ref{Calculation}, on the affine coordinate chart corresponding to $\theta_1\neq 0$ it has the expression
$$V^{*}= \left(\sum_{j=1}^{l}\,(h_{j1}\circ \pi)\,\xi_j\right)\,y_1\,\partial_{y_1} +
\sum_{j,i=1,\ i\neq 1}^{l}\,\left( (h_{ji}\circ \pi)\,\xi_j-(h_{j1}\circ \pi)\,\xi_j\,\xi_i\right)\,\partial_{\xi_i}$$
$$+ \sum_{i=1}^{k}\,(g_i\circ \pi)\,\partial_{x_i}, \qquad y_1\neq 0.$$
where we have defined $\xi_1=1$ for notational convenience.

In particular, it has a unique smooth extension to the hyperplane $E\cap [\theta_1\neq 0]=\mathcal{V}(y_1)$, it is tangent to it and restricted to it the field is the sum of the horizontal lifting of $V|_C$ and a vertical field.
Analogous calculations and conclusions hold for the other affine charts and we have the result.
\end{proof}

\begin{proof}[Proof of Lemma \ref{Pushout}]
Due to the vector field's tangency hypothesis, restricted to the affine chart corresponding to $\theta_1\neq 0$ with coordinates $x_1,\ldots, x_k, y_1,\xi_2,\ldots, \xi_l$ we have the expression:
$$V= h'_1\,\partial_{y_1} + \sum_{i=2}^{l}\,h_i\,\partial_{\xi_i} + \sum_{i=1}^{k}\,(g_i\circ \pi)\, \partial_{x_i}$$
where every coefficient is locally Lipschitz continuous and $h'_1$ is null on $\mathcal{V}(y_1)$.
\bigskip

\noindent\textit{Claim: There is a function $h_1$ bounded over compact sets such that $h'_1= h_1\,y_1$.}
\begin{proof}[Proof of the claim]
It is enough to verify this property on compact convex sets meeting $\mathcal{V}(y_1)$. Let $K$ be a compact convex set in the blowup meeting $\mathcal{V}(y_1)$ and denote by $k$ the Lipschitz constant of $h'_1$ over $K$. Then, over $K$ we have
$$|h'_1(x)|= |h'_1(x)- h'_1(x')|\leq k\,||x-x'||= k\, |y_1|$$
where $x$ is a point in $K$ and $x'$ is a point in $\mathcal{V}(y_1)\cap K$ with coordinates
$$x= (x_1,\ldots , x_k, y_1,\xi_2,\ldots, \xi_l),$$
$$x'= (x_1,\ldots , x_k, 0,\xi_2,\ldots, \xi_l)$$
whose difference is therefore $y_1$. Defining $h_1$ on $\mathcal{V}(y_1)\cap K$ by the zero constant, we conclude that $|h_1|$ is bounded by $k$ over $K$. Because the compact convex set was arbitrary, we have proved the claim.
\end{proof}

In resume, in the same affine chart as before we have the expression
$$V= h_1\,y_1\,\partial_{y_1} + \sum_{i=2}^{l}\,h_i\,\partial_{\xi_i} + \sum_{i=1}^{k}\,(g_i\circ \pi)\, \partial_{x_i}.$$
where every $h_j$ is bounded over compact sets. Therefore, by Lemma \ref{Calculation} the pushout in the region $y_1\neq 0$ is
$$V_*= (h_1\circ \pi^{-1})\,\sum_{i=1}^{l}\,y_i\,\partial_{y_i} + y_1\,\sum_{i=2}^{l}\,(h_i\circ \pi^{-1})\,\partial_{y_i} + \sum_{i=1}^{k}\,g_i\, \partial_{x_i}.$$
Let $\delta>0$ and consider the cube
$$R_1= \left[\,|y_1|\leq \delta\,\right]\cap \bigcap_{j,i}\,\left[\,|x_j|\leq \delta\,\right]\cap \left[\,|\xi_i|\leq 2\,\right]\subset Bl_C(M).$$
The compact set $\pi(R_1)$ fibers over $C\cap \pi(R_1)$ and the fibers are truncated cones with slope equals to two with respect to $y_1$. Analogous cones can be defined for the other charts and this choice for the slope is to guarantee that the union of all of these cones covers a neighbourhood of the portion of the center $C\cap \pi(R_1)$.

Note that on $\pi(R_1)$, $y_1\to 0$ implies $y_i\to 0$, for every $i=2,\ldots, l$ hence
$$V_*|_{\pi(R_1)-C}\xrightarrow{\ \  y_1\to 0 \ \ } \sum_{i=1}^{k}\,g_i\, \partial_{x_i}.$$
An analogous argument holds for the other affine charts giving compact regions $\pi(R_1),\ldots, \pi(R_l)$ that cover the cube
\begin{equation}\label{cube}
\bigcap_{j,i}\,\left[\,|x_j|\leq \delta\,\right]\cap \left[\,|y_i|\leq \delta\,\right]\subset M.
\end{equation}

In particular, the field $V_*$ continuously extends on $C\cap \pi(R_1)$ by defining it on the center with the expression
\begin{equation}\label{extension_V_definition}
V_*|_C= \sum_{i=1}^{k}\,g_i\, \partial_{x_i}.
\end{equation}
Because the cube \eqref{cube} was arbitrary, we conclude that $V_*$ continuously extends on $C$.

Now, we prove that the extension is actually locally Lipschitz. By hypothesis $V$ is locally Lipschitz continuous hence $V_*$ is so in $M-C$ because $\pi$ is an isomorphism in the complement of the exceptional divisor. It remains to show that this property holds on the center.

On the compact region $\pi(R_1)$ consider $x\in \pi(R_1)-C$ and $x'\in \pi(R_1)\cap C$ with coordinates
$$x=(x_1,\ldots , x_k, y_1,\ldots, y_l),$$
$$x'=(x'_1,\ldots , x'_k, 0,\ldots, 0).$$
We have
$$|(h_i\circ \pi^{-1})\,y_j| = |(h_i\circ \pi^{-1})|\,|\xi_j|\,|y_1|\leq 2\,\left(\max_{R_1}\,|h_i|\right)||x-x'||$$
because $y_1$ is a component of $x-x'$ where $x$ is in $\pi(R_1)-C$ and $x'$ is in $\pi(R_1)\cap C$. Then, every function $(h_i\circ \pi^{-1})\,y_j$ is Lipschitz on $\pi(R_1)$.

Every function $g_i\circ \pi$ is Lipschitz on $R_1$ with Lipschitz constant $k$. Hence
$$|g_i(x)-g_i(x')|= |g_i\circ \pi(z)-g_i\circ \pi(z')|\leq k\,||z-z'||$$
where $z$ is the preimage of $x$ by $\pi$
$$z= (x_1,\ldots , x_k, y_1,\xi_2,\ldots, \xi_l),$$
and we have chosen the preimage of $x'$ by $\pi$
$$z'= (x_1',\ldots , x_k', 0,\xi_2,\ldots, \xi_l).$$
In particular, $||z-z'||\leq ||x-x'||$ and we conclude that $g_i$ is Lipschitz on $\pi(R_1)$ for
$$|g_i(x)-g_i(x')|\leq k\,||x-x'||.$$

An analogous argument holds for the other regions $\pi(R_2),\ldots, \pi(R_l)$ and we conclude that $V_*$ is Lipschitz on the cube \eqref{cube} since each one of its coefficients is so. Again, because the cube \eqref{cube} was arbitrary, we conclude that $V_*$ is locally Lipschitz on $C$ and this finishes the proof.
\end{proof}

\subsection{Preliminaries on principalization}\label{section_preliminaries_principalization}

We refer to the classical book of Hartshorne \cite{Hart} and references therein for the notions of sheaves (section 1, chapter 2), ringed space (section 2, chapter 2), sheaf of ideals, direct and inverse image functor, total transform (section 5, chapter 2) and divisor (section 6, chapter 2). The notions in real algebraic geometry are completely analogous where the scheme must be replaced by the ringed space $(M,\,\I_M)$ (See subsection \ref{preliminaries_real_analytic}).

\begin{defi}\label{normal_crossing}
A divisor $E$ in $M$ has \textit{simple normal crossings with a variety $X$} if for every point $x$ in $X\cap E$ there is a coordinate neighbourhood $W_x$ of $x$ with real analytic coordinates $w_1,\ldots, w_n$ such that $X\cap W_x= \mathcal{V}(w_1,\ldots, w_k)$ and for every irreducible component $E'$ of $E$ containing $x$ there is a coordinate $w_j$ such that $E'\cap W_x= \mathcal{V}(w_j)$. A divisor $E$ in $M$ has \textit{simple normal crossings} if it has simple normal crossings with everyone of its points.
\end{defi}

The following is the celebrated \textit{Principalization Theorem}. Building upon work of H. Hironaka \cite{Hironaka}, O. Villamayor \cite{Villamayor1}, \cite{Villamayor2}, E. Bierstone and P. D. Milman \cite{Bierstone_Milman}, in \cite{Wlodarczyk} J. W\l{}odarczyk proved a locally finite functorial principalization theorem for real and complex analytic spaces. We will be concerned with its application to real analytic spaces.

\begin{teo}[J. W\l{}odarczyk, \cite{Wlodarczyk} Theorem 2.0.3]\label{posta}
Let $\mathcal{I}$ be a coherent ideal sheaf of the sheaf $\I_M$ on an analytic manifold $M$ (not necessarily
compact). There exists a locally finite principalization of $\mathcal{I}$, that is to say, a manifold $\tilde{M}$, a
proper morphism
$$\textsf{prin}_{\mathcal{I}}: \tilde{M}\rightarrow M$$
and an ideal sheaf $\tilde{\mathcal{I}}$ of the sheaf $\I_{\tilde{M}}$ on $\tilde M$ such that
\begin{enumerate}
\item For any compact subset $Z$ in $M$ there are open neighbourhoods $A\supset Z$ in $M$ and $\tilde{A}=\textsf{prin}_{\mathcal{I}}^{-1}(A)$ in $\tilde{M}$ for which the restriction $\textsf{prin}_{\mathcal{I}}|_{\tilde{A}}$ splits into a finite sequence of blowups
\begin{equation}\label{splitting}
\textsf{prin}_{\mathcal{I}}|_{\tilde{A}}\,:\,\tilde{A}=A_r\xrightarrow{\pi_r} \ldots\rightarrow A_2\xrightarrow{\pi_2} A_1\xrightarrow{\pi_1} A_0= A
\end{equation}
with smooth centers $C_i\subset A_i$ such that

\item\label{SNC_item} The exceptional divisor $E_{A_i}$ of the induced morphism
$$\pi^i=\pi_1\circ\ldots\circ\pi_i:A_i\rightarrow A$$
has only simple normal crossings and $C_i$ has simple normal crossings with $E_{A_i}$.

\item\label{inverse_image_functor} The total transform $(\textsf{prin}_{\mathcal{I}}|_{\tilde{A}})^*(\mathcal{I}|_A)$ is the ideal sheaf of a simple normal crossing divisor $\tilde{E}_A$ which is a locally finite combination of the irreducible components of the divisor $E_{A_r}$.

\item\label{funct_item} For any pair of compact sets $Z'\subset Z$ and corresponding open neighbourhoods $A'\subset A$, omitting the empty blowups the restriction of the factorization \eqref{splitting} of $\textsf{prin}_{\mathcal{I}}|_{\tilde{A}}$ to $\tilde{A}'$ determines the factorization of $\textsf{prin}_{\mathcal{I}}|_{\tilde{A}'}$.
\end{enumerate}
The morphism $\textsf{prin}: (\tilde{M}, \tilde{\mathcal{I}})\rightarrow (M, \mathcal{I})$ commutes with local analytic isomorphisms and embeddings of ambient varieties.
\end{teo}

We will refer to item  \ref{funct_item} and the last sentence of Theorem \ref{posta} as the \textit{functoriality of the pricipalization}. By the first item, it is clear that $\tilde{M}$ is a real analytic manifold. 

\begin{prop}\label{funct}
Let $\mathcal{I}$ be a coherent ideal sheaf on an analytic manifold $M$ and consider the principalization $\textsf{prin}_{\mathcal{I}}: \tilde{M}\rightarrow M$ in Theorem \ref{posta}. Then,
\begin{enumerate}
\item\label{precompact_item}
For every precompact neighbourhood $B$ in $M$, the restricted principalization $\textsf{prin}_{\mathcal{I}}|_{\tilde{B}}$ splits into a finite blowup sequence.
\item\label{pegado_local_global}
For every pair of open sets $A$ and $A'$ in $M$ satisfying the statement in \ref{posta} we have
$$\tilde{E}_{A\cap A'}\,=\,\tilde{E}_A\cap \tilde{E}_{A'}.$$
\item
The principalization $\textsf{prin}_{\mathcal{I}}$ has an exceptional divisor $\tilde{E}$ in $\tilde{M}$ with only simple normal crossings such that
$$\tilde{E}\cap \tilde{A}\,=\,\tilde{E}_A$$
for every open set $A$ in $M$ satisfying the statement in \ref{posta}.
\end{enumerate}
\end{prop}
\begin{proof}
The first item immediately follows from the first item of \ref{posta} and the functoriality. In effect, by Theorem \ref{posta} there is a neighbourhood $A$ of the compact subset $Z=\bar{B}$ whereat the pricipalization splits into a blowup sequence and because $B\subset A$ is an embedding, by functoriality we have the result.

For the second item note that the first item in \ref{posta} gives the commutative diagram  
\begin{equation*}
\xymatrix@C+=2cm{
    \tilde{A}\cap \tilde{A}'=\widetilde{A\cap A'} \ar[rr]^{\textit{\textsf{prin}}_{\mathcal{I}}|_{\widetilde{A\cap A'}}} \ar[d]_{\tilde{\iota}} & & A\cap A'\ar[d]^{\iota} \\
  \tilde  A \ar[rr]_{\textit{\textsf{prin}}_{\mathcal{I}}|_{\widetilde{A}}}    & & A }
\end{equation*}
where $\iota$ and $\tilde{\iota}$ are the respective canonical inclusions. From this diagram it follows that 
$$\left(\textit{\textsf{prin}}_{\mathcal{I}}|_{\widetilde{A\cap A'}}\right)^{\ast}\left(\mathcal{I}|_{A\cap A'}\right)\,=\,  \left(\textit{\textsf{prin}}_{\mathcal{I}}|_{\widetilde{A}}\right)^{\ast}\left(\mathcal{I}|_A\right)|_{ 
\widetilde{A\cap A'}}.$$

The lefthand side of the previous identity is the ideal sheaf  of $\tilde E_{A\cap A'}$ while the righthand side is the ideal sheaf of $\tilde E_{A}\cap \widetilde{A\cap A'}$ hence these divisors coincide, that is to say
$$\tilde E_{A\cap A'}\,=\,\tilde E_{A}\cap \widetilde{A\cap A'}\,=\,\tilde E_{A}\cap \tilde{A}'.$$

The same argument applied on $A'$ instead gives the identity
$$\tilde E_{A\cap A'}\,=\,\tilde E_{A'}\cap \widetilde{A\cap A'}\,=\,\tilde E_{A'}\cap \tilde{A}$$
and because the intersection of $\tilde E_{A'}\cap \widetilde{A\cap A'}$ and $\tilde E_{A}\cap \widetilde{A\cap A'}$ is 
$\tilde E_{A'}\cap \tilde E_A$, this finishes the proof of the second item.

The third item follows directly from the second one and item \ref{SNC_item} in Theorem \ref{posta}. This finishes the proof of the proposition.
\end{proof}

Whenever the ideal sheaf we are working with were clear from the context, we will refer to the principalization of $(A,\,\mathcal{I}|_A)$ simply as the \textit{principalization on $A$}.

The principalization algorithm in the proof of Theorem \ref{posta} does not do anything on open sets $A\subset M$ such that $\mathcal{I}|_A=\I_A$ because the order of the ideal sheaf at every point in $A$ is zero, that is to say for every $x$ in $A$,
$$\mbox{ord}_x(\mathcal{I})\,=\,\max\,\lbrace\,i\ |\ \mathcal{I}_x\subset \mathfrak{m}_x^i\,\rbrace\,=\,0.$$

The exceptional divisor $E_i$ in $A_i$ is defined inductively as
\begin{equation}\label{exceptional_induction_definition}
E_i= \pi_i^*(E_{i-1})\cup \pi_i^{-1}(C_{i-1})
\end{equation}
where the first factor on the right hand side denotes the strict transform. It is clear from the definition that if the center has codimension one, then the corrsponding blowup is an isomorphism whose only purpose is to turn the center into a new irreducible component of the exceptional divisor.

As usual, we order the (not necessarily irreducible) components $H_{ij}$ of the exceptional divisor $E_i$ in $A_i$ in order of appearance, that is to say by their date of birth:
\begin{subequations}\label{general_structure}
\begin{align}
\label{gen_str_1}& E_i\,=\,H_{ii}\cup H_{i,\,i-1}\cup\ldots \cup H_{i1},\\
\label{gen_str_2}& H_{ii}\,=\,\overline{\pi_i^{-1}(C_{i-1})\,-\,\pi_i^*(E_{i-1})},\\
\label{gen_str_3}& \pi_i|_{H_{ij}}:H_{ij}\xrightarrow{\cong} H_{i-1,\, j},\quad j<i.
\end{align}
\end{subequations}
Note that $H_{ii}$ is the exceptional divisor of the single blowup $\pi_i$. In particular, by expression \ref{gen_str_3}, every component $H_{ij}$ is smooth.

The splitting of the smooth center $C_{j-1}$ into its irreducible components
$$C_{j-1}\,=\,C_{j-1}^1\sqcup\,\ldots\,\sqcup C_{j-1}^l$$
induce an splitting by \eqref{gen_str_2} of $H_{jj}$ into its irreducible components as well
$$H_{jj}\,=\,H_{jj}^1\sqcup\,\ldots\,\sqcup H_{jj}^l$$
and this splitting is preserved under the isomorphism \eqref{gen_str_3},
$$H_{ij}\,=\,H_{ij}^1\sqcup\,\ldots\,\sqcup H_{ij}^l,\quad j<i.$$

In particular, every irreducible component $H_{ij}^k$ of the exceptional divisor $E_i$ is closed since the centers are closed as well by definition. For notational convenience, we define the real analytic proper maps
$$\pi^{i}_j = \pi_{j+1}\circ \ldots \circ \pi_i: A_i\rightarrow A_j.$$
Note the isomorphism between the irreducible components
\begin{equation}\label{isomorphism_irr_components}
\pi^i_j|_{H_{ij}^k}:H_{ij}^k\xrightarrow{\cong} H_{jj}^k,\quad j<i.
\end{equation}
Then, we can always identify the irreducible component $H_{ij}^k$ in $E_{i}$ with the corresponding $H_{jj}^k$ in $E_j$.

\begin{prop}[Embedded desingularization]\label{embedded_des}
Let $X'$ be an irreducible component of a closed analytic subspace $X$. Then, for every precompact neighbourhood $B$ in $M$ there is a minimum index $j$ such that the $j$-th strict transform $X'_j$ of $X'_0=B\cap X'$ is contained in the center $C_j$ and it is actually an irreducible component of it. In particular, $X'_j$ is smooth and has simple normal crossings with the exceptional divisor $E_{B_j}$.
\end{prop}
\begin{proof}
After a similar argument as in item \ref{precompact_item} in the previous proposition \ref{funct}, the proof may be found in \cite{Wlodarczyk}, section 4, pages 42-43.
\end{proof}

The principalization may take several steps after the desingularization. As an example, consider the polynomial $f$ in $\R[x,y,z]$ defined by
$$f(x,y,z)=x^6+y^2.$$
Its zero locus is the $z$-axis which is a smooth submanifold. In particular, the embedded desingularization is trivial and is realized in the zeroth step. However, the principalization is achieved after three steps: the first step is the blowup
$$y=y_1\,x,\qquad f_1(x,y_1,z)\,=\,x^2\,(x^4+y_1^2),$$
the second step is the blowup
$$y_1=y_2\,x,\qquad f_2(x,y_2,z)\,=\,x^4\,(x^2+y_2^2)$$
and the third step is the blowup
$$y_2=y_3\,x,\qquad f_3(x,y_3,z)\,=\,x^6\,(1+y_3^2).$$
Defining the analytic coordinates $w_1=x\,(1+y_3^2)^{1/6}$, $w_2=y_3$ and $w_3=z$, we have the monomial form
$$f_3(w_1,w_2,w_3)\,=\,w_1^6.$$

\subsection{Proof of Theorems \ref{Theorem_codim1} and \ref{Theorem_codim2}}

Consider a real analytic potential $U$ whose zero locus is nonempty and consider a regular component $X$ of the zero locus $\mathcal{V}(U)$, that is to say a connected component of $\mathring{\mathcal{V}(U)}$ which is an open subset of $\mathcal{V}(U)$ and a submanifold of $M$ by item \ref{open_submanifold} in Theorem \ref{smooth_points}. Because $X$ is a relative open set of the zero locus, there is an open set $A$ in $M$ such that $A\cap \mathcal{V}(U)=X$. Therefore, by restricting to the open set $A$, \textbf{without loss of generality we may suppose that $\mathcal{V}(U)\,=\,X$ and $X$ is smooth and connected.}

The potential $U$ defines a ringed space $(M,\,\mathcal{I}_U)$ where the coherent ideal sheaf is the one associated with the real analytic potential $U$, that is to say the ideal sheaf whose stalk at the point $p$ is the localization of the ideal $(U)\subset \I_M$ at the point $p$:
\begin{equation}\label{rmrk0}
\mathcal{I}_{U,\,p}= (U)_p\subset \I_{M,\,p},\qquad p\in M.
\end{equation}

Although the following results follow directly from the principalization algorithm, in the following lemma we deduce them from the results in the previous section \ref{section_preliminaries_principalization}.

\begin{lema}\label{principalization_algorithm_bye}
Consider a precompact neighbourhood $A$ in $M$ and the principalization of the ideal sheaf $\mathcal{I}_{U}$ on $A$. Then, the following assertions hold:
\begin{enumerate}
\item If $A\cap X$ is empty, then $\textsf{prin}_{\mathcal{I}_U}|_{\tilde{A}}$ is the identity map.
\item\label{center_contained_in_XE} For every $i\geq 0$, $C_i\subset X_i\cup E_i$ where $X_i$ denotes the $i$-th strict transform of $X_0=A\cap X$.
\item\label{preimage_of_X0} For every $i\geq 1$, $(\pi^i)^{-1}(X\cap A)\,=\, X_i\cup E_i$.
\end{enumerate}
\end{lema}
\begin{proof}
Because $A$ does not meet $X$ we have that $U$ is a non vanishing real analytic function on $A$ hence $\mathcal{I}_U|_A=\I_A$. Then, as it was mentioned in the previous section \ref{section_preliminaries_principalization}, the order of the ideal sheaf $\mathcal{I}_U$ is zero at every point in $A$ hence the principalization algorithm does not do anything on $A$. This concludes the proof of the first item.

Suppose that there is an index $i\geq 0$ such that $C_i$ is not contained in $X_i\cup E_i$ and let it be the minimum index with this property. Define $\pi^0$ simply as the identity and recall that $E_0$ is empty.

We claim that $(\pi^k)^{-1}(X\cap A)\,=\, X_k\cup E_k$ for every $0\leq k \leq i$. Indeed, it is trivially checked for $k=0$ and assuming that it holds for $k-1$, by the definition of the strict transform \eqref{strict_transform_definition} and expression \eqref{exceptional_induction_definition}, we have
$$X_k\cup E_k\,=\,\overline{\pi_k^{-1}(X_{k-1}-C_{k-1})}\,\cup\,\overline{\pi_k^{-1}(E_{k-1}-C_{k-1})}\,\cup\,\pi_k^{-1}(C_{k-1})$$
hence after some manipulation
\begin{equation}\label{aux}
\pi_k^{-1}(X_{k-1}\cup E_{k-1})\,\subset\,X_k\cup E_k\,\subset\,\overline{\pi_k^{-1}(X_{k-1}\cup E_{k-1})}.
\end{equation}
Recall that $X=\mathcal{V}(U)$ hence it is closed. By the inductive hypothesis and the fact that $X\cap A$ is relatively closed in $A$, we have that $X_{k-1}\cup E_{k-1}$ is relatively closed in $A_{k-1}$ hence the left hand side of \eqref{aux} is relatively closed in $A_k$ and we have proved the induction step
$$\pi_k^{-1}(X_{k-1}\cup E_{k-1})\,=\,X_k\cup E_k.$$
This finishes the proof of the claim.

Consider a point $p$ in $C_i-(X_i\cup E_i)$ and a precompact neighbourhood $B'$ of $p$ in $A_i$ such that
$$\overline{B'}\subset A_i\qquad \mbox{and}\qquad\overline{B'}\cap (X_i\cap E_i)=\emptyset.$$
Define $B=\pi^i(B')$. Since $\pi^i$ is an analytic isomorphism outside $E_i$ and
$$(\pi^i)^{-1}(X\cap A)\,=\, X_i\cup E_i$$
by the claim we just proved, we have that $B$ is a precompact neighbourhood which does not meet $X$. Therefore, by the previous item, the principalization on $B$ consist of a sequence of empty blowups which is absurd by item \ref{funct_item} in Theorem \ref{posta} because the $i$-th center of the restriction on $B$ of the principalization on $A$ is $C_i\cap B_i$ which is nonempty. Here $B_i$ is the $i$-th blowup space in the blowup sequence of $B$. This concludes the proof of the second item.

The proof of the third item is verbatim as the one of the claim and this finishes the proof of the lemma.
\end{proof}

\begin{cor}\label{property2}
Consider the functorial principalization $\textit{\textsf{prin}}_{\mathcal{I}_U}: \tilde{M}\rightarrow M$. Then,
$$\textit{\textsf{prin}}_{\mathcal{I}_U}^{-1}(X)\,=\,\textit{\textsf{prin}}_{\mathcal{I}_U}^{-1}(\mathcal{V}(U))\,=\, \tilde{E}.$$
\end{cor}
\begin{proof}
Because $M$ is locally compact and the result is local, it is sufficient to prove the expression locally on every precompact neighbourhood in $M$. This follows directly from item \ref{preimage_of_X0} in Lemma \ref{principalization_algorithm_bye} applied at the top index $r$ of the respective blowup sequence of the principalization and the fact $X_r$ is empty by Proposition \ref{embedded_des}. This concludes the proof.
\end{proof}

Unless otherwise stated, along the proof the ideal sheaf considered in the principalization will allways be $\mathcal{I}_U$ and its restrictions.

\begin{lema}\label{rmrk-1}
Consider the real analytic function $\tilde{U}=U\circ\textit{\textsf{prin}}_{\mathcal{I}_U}$.
\begin{enumerate}
\item\label{zero_locus_U_tilde} The zero locus of $\tilde{U}$ is the exceptional divisor $\tilde{E}$, that is to say $\mathcal{V}(\tilde{U})=\tilde{E}$.
\item\label{monomialization} For every point in the exceptional divisor $\tilde{E}$ there is a coordinate neighborhood $W$ of the point in $\tilde{M}$ with real analytic coordinates $w_1,\ldots, w_n$ such that:
$$\tilde{U}|_W\,  =\, \pm w_1^{d_1}\ldots w_n^{d_n}$$
for some positive integer $d_1$ and some nonnegative integers $d_2,\ldots, d_n$.
\end{enumerate}
\end{lema}
\begin{proof}
The first item immediately follows from Corollary \ref{property2}. Indeed,
$$\mathcal{V}(\,\tilde{U})=\textit{\textsf{prin}}_{\mathcal{I}_U}^{-1}(\mathcal{V}(U))=\textit{\textsf{prin}}_{\mathcal{I}_U}^{-1}(X)=\tilde{E}.$$

For the second item first note that $\tilde{U}$ can only vanish in $\tilde E$ on account of the first item.  
Because $\tilde{E}$ has only normal crossings, by definition \ref{normal_crossing}, the fact that $\tilde{U}$ is a real analytic function and the first item, we have that for every point in the exceptional divisor $\tilde{E}$ there is a coordinate neighborhood $W$ of the point in $\tilde{M}$ such that:
$$\tilde{U}|_W\,  =\, h\, w_1^{d_1}\ldots w_n^{d_n}$$
for some unit $h$ in $\I_W$ , some positive integer $d_1$ and some nonnegative integers $d_2,\ldots, d_n$ where $w_1,\ldots, w_n$ are the real analytic coordinates of $W$. The unit $h$ is a real analytic non vanishing function on $W$ and without loss of generality we may suppose that it is positive. Then, defining the new coordinate $w_1'=(h)^{1/d_1}\,w_1$ we get the result.
\end{proof}

\subsubsection{Proof of Theorem \ref{Theorem_codim1}}\label{Codim1}

\begin{lema}\label{conjunto_S}
Suppose that $X$ has codimension one in $M$. Then, there is a closed nowhere dense zero Lebesgue measure subset $S$ in $X$ such that the principalization is an isomorphism on $M-S$ and after identifying this space with the corresponding one of the principalization, the exceptional divisor coincides with $X-S$ on $M-S$,
$$\tilde{E}|_{M-S}\,=\,X-S.$$
Moreover, on every precompact neighbourhood in $X$, the number of connected components of $S$ is finite.
\end{lema}
\begin{proof}
For every precompact neighbourhood $A$ in $M$, by item \ref{preimage_of_X0} in Lemma \ref{principalization_algorithm_bye}, we have real analytic proper maps
\begin{equation}\label{center_maps}
\phi_{A,i,j}\,=\,\pi^{i}|_{C_i^j}:C_i^j \rightarrow A\cap X
\end{equation}
where $C_i^j$ denotes the $j$-th irreducible component of the center $C_i$. By Corollary \ref{Sard}, the set $S_{A,i,j}$ of critical values of the map $\phi_{A,i,j}$ is a closed nowhere dense zero Lebesgue measure subset in $A\cap X$.

For every precompact neighbourhood $A$ in $M$, define the set $S_A$ as the union of all the critical sets $S_{A,i,j}$ such that $C_i^j$ has codimension at least two:
$$S_A\,=\,\bigcup_{codim\ C_i^j \geq 2}\,S_{A,i,j}.$$
Because it is a finite union, it is clear that $S_A$ is a closed nowhere dense zero Lebesgue measure subset in $A\cap X$. Because of the functoriality, recall that this notion includes item \ref{funct_item} in Theorem \ref{posta}, by a similiar argument as in item \ref{pegado_local_global} in Proposition \ref{funct} we have the relation
$$S_{A\cap A'}\,=\,S_A\cap S_{A'}$$
hence, because $M$ is locally compact and second axiom, these sets define a closed nowhere dense zero Lebesgue measure subset $S$ in $X$ such that $S\cap A=S_A$ for every precompact neighbourhood $A$ in $M$.

Because $X$ has codimension one, every $S_{A,i,j}$ is the image of the map $\phi_{A,i,j}$ which is connected hence the set of connected components of $S_A$ is finite since the union is finite.

For every precompact neighbourhood $A$ in $M-S$, by the functoriality of the principalization and omitting the empty blowups, the restricton of the blowup sequence into $A$ gives a principalization whose blowups has codimension one centers. These blowups are isomorphisms whose only purpose is to turn the center into a new irreducible component of the exceptional divisor $\tilde{E}_A$. Then, identifying the spaces $\tilde{A}$ and $A$ in such a way that the principalization map is the identity, by Corollary \ref{property2} we have that $\tilde{E}_A = A\cap (X-S)$ and because the precompact neighbourhood $A$ was arbitrary and $M-S$ is locally compact, we conclude that $\tilde{E}= X-S$ in $M-S$ and we have the result.
\end{proof}

\begin{proof}[Proof of Theorem \ref{Theorem_codim1}]
Suppose that the real analytic (smooth) vector field $V$ is defined on the open subset $B\subset M$ and is tangent to $X$ at the possibly empty intersection $B\cap X$. On $B-X$, there exists a unique real analytic (smooth) function $P$ defined by
$$V(U)= P\,U.$$
Consider the set $S$ of Lemma \ref{conjunto_S}. The rest of the proof consist in the construction of a real analytic (smooth) extension of $P$ on $B-S$.

If $(B\cap X)$ is empty, then there is nothing to prove. Otherwise, there are points in $(B\cap X)-S$ since $S$ is nowhere dense in $X$ hence in $B\cap X$ as well. Consider a point $p$ in $(B\cap X)-S$. By Lemma \ref{conjunto_S} and the functoriality of the principalization, we have
$$\tilde{E}|_{B-S}=(B\cap X)-S$$
hence by Lemma \ref{rmrk-1}, there is a coordinate neighbourhood $W_p$ of $p$ in $B-S$ with real analytic coordinates $w_1,\,w_2,\ldots,\, w_n$ such that:
$$U|_{W_p}\,  =\, \pm w_1^{d_1}\ldots w_n^{d_n}$$
for some positive integer $d_1$ and some nonnegative integers $d_2,\ldots, d_n$. If any of the integers $d_2,\ldots, d_n$ were positive, then there would be a singular point of $X$ which is absurd because $X$ is a regular component. Therefore we have
$$U|_{W_p}\,  =\, \pm w_1^{d_1},\qquad\quad d_1>0.$$

Because $V$ is real analytic (smooth) and tangent to $X$ at $B\cap X$, with respect to these coordinates we have
$$V|_{W_p}= h_1\,w_1\ \partial_{w_1}+ h_2\,\partial_{w_2} + \ldots + h_n\,\partial_{w_n}$$
for some real analytic (smooth) functions $h_1,\ldots, h_n$. Therefore, $P$ extends real analytically (smoothly) on $W_p$ just by defining it over this neighbourhood, with respect to the coordinates of $W_p$, by the expression
$$P|_{W_p}=  d_1\,h_1.$$
Because the point $p$ was arbitrary, we have the result.
\end{proof}

\subsubsection{Proof of Theorem \ref{Theorem_codim2}}\label{Codim2}$\,$\\

The strategy for the proof of  Theorem \ref{Theorem_codim2} can be split into two parts. First we  obtain  regions which admit  nice principalizations having regularity  properties 
described in detail in  Lemma \ref{limpiando_la_cancha}.  In the rest of this subsection we show how  these properties allow us to apply systematically the posibility of  taking  pullbacks and pushouts according to Lemmas   \ref{Pullback} and \ref{Pushout}.

\begin{lema}\label{extension}
Consider a submanifold $N\subset M$ and a vector $V$ field defined on it. Then, there is a neighbourhood $W$ of $N$ in $M$ and a vector field $\hat{V}$ defined on $W$ extending $V$.
\end{lema}
\begin{proof}
This is an elementary construction. For every point $p$ in $N$ there is a coordinate neighbourhood $(W_p,\,\varphi_p)$ of $p$ in $M$ such that
$$\varphi(W_p\cap N)\,=\,\varphi(W_p)\cap [x_{k+1},\,\ldots,\,x_n\,=\,0]$$
where $x_1,\ldots,\,x_n$ are the canonical coordinates of $\R^n$. Define the vector field $\hat{v}_p$ simply as the translation of $(\varphi_p)_*(V)$ along the last $n-k$ coordinates, that is to say
$$\hat{v}_p(x_1,\ldots,\,x_n)\,=\,(\varphi_p)_*(V)(x_1,\ldots,\,x_k,\,0,\ldots,\,0).$$
Define the vector field $\hat{V}_p=(\varphi_p^{-1})_*(\hat{v}_p)$ on some possibly smaller neighbourhood $W_p'\subset W_p$ of $p$ in $M$ and note that it extends $V$ on this neighbourhood. Define
$$W\,=\,\bigcup_{p\in N}\,W_p'$$
and consider a partition of unity $\{f_p\ |\ p\in N\}$ subordinated to the open cover $\{W_p'\ |\ p\in N\}$ of $W$. Defining
$$\hat{V}\,=\,\sum_{p\in N}\,f_p\,\hat{V}_p$$
we have the result.
\end{proof}

\begin{lema}\label{extension2}
Consider a codimension two closed real analytic submanifold $C$ in $M$ and the blowup along this center $\pi: Bl_C(M)\rightarrow M$. Consider a smooth section $\sigma$ of the fibration $\pi^{-1}(C)\rightarrow C$ and a smooth vertical vector field $V$ along the image of $\sigma$, that is to say,
$$d_{\sigma(a)}\pi\,(\,V_{\sigma(a)})\,=\,{\bf 0},\qquad a\in C.$$
Then, there is a neighbourhood $W$ of $\pi^{-1}(C)$ in $Bl_C(M)$ and a vector field $\hat{V}$ defined on $W$ such that it is vertical on $\pi^{-1}(C)$ and extends $V$.
\end{lema}
\begin{proof}
The fibration $\pi^{-1}(C)\rightarrow C$ is a $G$-principal fibration with group
$$G\,=\,P^1(\R)\,=\,U(1)/\{\pm 1\}.$$
Then there is a vertical smooth extension $\tilde{V}$ defined on $\pi^{-1}(C)$ given by
$$\tilde{V}_{g\cdot \sigma(a)}\,=\,(l_g)_*(\,V_{\sigma(a)}),\qquad a\in C$$
where $l_g$ denotes the left action of $g$ in $P^1(\R)$ on the fibration. The extension is well defined because the group action is free. By Lemma \ref{extension}, there is a neighbourhood $W$ of $\pi^{-1}(C)$ in $Bl_C(M)$ and a vector field $\hat{V}$ defined on $W$ extending $\tilde{V}$ and we have the result.
\end{proof}

\begin{lema}\label{limpiando_1}
Consider a precompact neighbourhood $A$ in $M$ and the principalization on $A$. If an irreducible component of the center $C_i$ has codimension one, then it coincides with some irreducible component of the exceptional divisor $E_i$.
\end{lema}
\begin{proof}
We have $C_{i}^k\subset C_i\subset X_i\cup E_i$ by item \ref{center_contained_in_XE} in Lemma \ref{principalization_algorithm_bye}. The $i$-th strict transform $X_i$ of $X_0=A\cap X$ is smooth and has codimension two since $X_0$ has the same properties and the strict transform preserves the smoothness and dimension of points. Hence by dimension reasons $C_{i}^k$ must be fully contained in $E_i$. In effect, if there is a point $p$ in $X_i- E_i$ such that $C_{i,p}^k\subset X_{i,p}$, then by Theorem \ref{fully_contained} $C_{i}^k\subset X_{i}$ which is absurd by dimension.


Recall that $C_{i}^k\subset E_i$ and $C_{i}^k$ has only simple normal crossings with $E_i$. By definition \ref{normal_crossing}, for any point $q$ in $C_{i}^k$ there is a coordinate neighbourhood $W_q$ of $q$ with real analytic coordinates $w_1,\ldots, w_n$ such that $C_{i}^k\cap W_x= \mathcal{V}(w_1)$ and for every irreducible component $E'$ of $E_i$ containing $q$ there is a coordinate $w_j=w(E')$ such that $E'\cap W_x= \mathcal{V}(w_j)$. Then, since $C_{i}^k\subset E_i$ we have
$$\mathcal{V}(w_1)\,\subset\,\bigcup_{E'<E_i}\,\mathcal{V}(w(E'))$$
and taking ideals we have
$$(w_1)\,\supset\,\left(\prod_{E'<E_i}\,w(E')\right)$$
hence $w_1$ must divide the product of coordinates
$$w_1\ |\ \prod_{E'<E_i}\,w(E')$$
where we have denoted by $E'<E_i$ the irreducible component $E'$ of $E_i$ containing $q$. Thus, $w_1$ must coincide with some $w(E')$ with $E'<E_i$ and we conclude that there is an irreducible component $E'$ of $E_i$ containing $q$ such that the germs $C_{i,q}^k= E'_q$. By Theorem \ref{fully_contained} again we have the equality $C_{i}^k= E'$ and we have the result.
\end{proof}

\begin{lema}\label{transversality}
An irreducible component $C_i^j$ of the center $C_i$ is transversal to an irreducible component $H_{ik}^m$ of the exceptional divisor $E_i$ if and only if $C_i^j$ is not contained in $H_{ik}^m$.
\end{lema}
\begin{proof}
Consider the irreducible components $C_i^j$ and $H_{ik}^m$ of $C_i$ and $E_i$ respectively. If their intersection is empty, then the result is trivially verified. Suppose that there is a point $p$ in their intersection. Since $C_i$ and $E_i$ have only simple normal crossings, by definition \ref{normal_crossing} there is a coordinate neighbourhood $W_p$ of $p$ in $A_i$ with real analytic coordinates $w_1,\ldots, w_n$ such that $C_i^j\cap W_x= \mathcal{V}(w_1,\ldots, w_r)$ and $H_{ik}^m\cap W_x= \mathcal{V}(w_s)$. Then their respective tangent spaces at $p$ read as follows
$$T_p\,C_i^j\,=\,\langle\,\partial_{w_{r+1}}|_p,\ldots,\partial_{w_{n}}|_p\,\rangle,$$
$$T_p\,H_{ik}^m\,=\,\langle\,\partial_{w_{1}}|_p,\ldots,\partial_{w_{s-1}}|_p,\,\partial_{w_{s+1}}|_p,\ldots,\partial_{w_{n}}|_p\,\rangle.$$
Thus, the intersection is transversal if and only if $w_s$ is not among the first $r$ coordinates, that is
$$T_p\,C_i^j\,+\,T_p\,H_{ik}^m\,=\,T_p A_i\qquad\mbox{iff}\qquad s>r.$$
By Theorem \ref{fully_contained}, $C_i^j\subset H_{ik}^m$ if and only if the germs $C_{i,\,p}^j\subset H_{ik,\,p}^m$ and this is verified if and only if $s\leq r$. This concludes the proof of the lemma.
\end{proof}

\begin{lema}\label{limpiando_la_cancha}
Suppose that $X$ has codimension two and $F$ is a smooth vector field on $X$. Then, there is a closed nowhere dense zero Lebesgue measure subset $S$ in $X$ such that for every point $p$ in $X-S$ there is a neighbourhood $W_p$ of $p$ in $M$ and a finite principalization on $W_p$ whose blowup sequence satisfies:
\begin{enumerate}
\item \label{field_extension} There is a smooth extension of the vector field $F$ to $W_p$.
\item \label{first_center} $C_0= W_p\cap X$. In particular, the strict transforms are empty and $C_i\subset E_i$ for every $i\geq 1$.
\item\label{maps_isomorphisms} For every $i$, the center $C_i$ is irreducible and the map $\pi^{i}|_{C_i}: C_i\rightarrow W_p\cap X$ is an isomorphism.
\item\label{fully_contained_item} If the center $C_i$ intersects some irreducible component of the exceptional divisor $E_i$, then it is fully contained in this component.
\item \label{regular_sequences} Every center $C_i$ is the zero locus of some strong regular sequence $\mathcal{R}_i$.
\end{enumerate}
\end{lema}
\begin{proof}
For every precompact neighbourhood $A$ in $M$, consider the principalization on $A$ and the respective $j$ index in Proposition \ref{embedded_des}. Suppose that $A$ is small enough such that $A\cap X$ is irreducible.

Our \textbf{zeroth claim} is that, without loss of generality, we may assume that every irreducible component of the centers has codimension greater than or equal to two. In effect, because of Lemma \ref{limpiando_1}, if there is a codimension one irreducible component $C_{i}^k$ of the center $C_i$, then it coincides with some irreducible component $H_{il}^m$ of the exceptional divisor $E_i$. Then, the blowup along $C_{i}^k$ is an isomophism and by \eqref{gen_str_2} and \eqref{isomorphism_irr_components} we have
$$H_{i+1,\,i+1}^k\,=\,H_{i+1,\,l}^m.$$
In particular, the blowup along $C_{i}^k$ is an isomorphism adding no new irreducible component to the exceptional divisor hence it could be omitted since it would have the same effect as an empty blowup.



Our \textbf{first claim} is that $X_j$, the $j$-th strict transform of $X_0=A\cap X$, is not contained in any irreducible component $H_{ji}^m$, $i\leq j$ of $E_j$. Indeed, suppose that there is an index $i\leq j$ such that $X_j\subset H_{ji}^m$. Then by the isomorphism \eqref{isomorphism_irr_components}, $X_i\subset H_{ii}^m$ hence projecting with $\pi_i$ and recalling \eqref{gen_str_2}, we conclude that $X_{i-1}\subset C_{i-1}^m$ which is absurd since $j$ is the minimum index with this property.

Our \textbf{second claim} is that $X_j$ is transversal to any irreducible component $H_{ji}^m$, $i\leq j$ of $E_j$. This follows directly from the first claim and Lemma \ref{transversality} since by Proposition \ref{embedded_des}, $X_j$ is an irreducible component of $C_j$.

Therefore, if $X_j\cap H_{ji}^m$ is non empty, then it is a codimension greater than two submanifold hence, by a similar argument as in Corollary \ref{Sard}, its image under $\pi^j$ is a closed nowhere dense zero Lebesgue measure subset in $A\cap X$ and so is the finite union
$$S_A^1=\bigcup_{i\leq j,\,m}\,\pi^j(X_j\cap H_{ji}^m)= \pi^j(X_j\cap E_j).$$
By the functoriality of the principalization,
recall that this notion includes item \ref{funct_item} in Theorem \ref{posta}, by a similiar argument as in item \ref{pegado_local_global} in Proposition \ref{funct}, for any other precompact neighbourhood $A'$ in $M$ we have
$$S_{A\cap A'}^1\,=\,S_A^1\cap S_{A'}^1.$$
Hence, because $M$ is locally compact and second axiom, these sets define a closed nowhere dense zero Lebesgue measure subset $S^1$ in $X$ such that $S^1\cap A= S_A^1$ for every precompact neighbourhood $A$ in $M$.

Our \textbf{third claim} is that, omitting the empty blowups, the principalization restricted to any precompact open set $A$ in $M-S^1$ with nonempty intersection with $X$ has first center $C_0=A\cap X$.

Indeed, suppose that there is an index $0\leq i < j$ such that $C_i$ is nonempty and let it be the minimum index with this property. Then, by item \ref{center_contained_in_XE} in Lemma \ref{principalization_algorithm_bye} and the fact that $E_i$ must be the empty set by equations \eqref{general_structure} because there is no previous non empty centers to generate it, we have that $C_i\subset X_i$ hence $H_{ji}^m\cap X_j$ is nonempty for some irreducible component $H_{ji}^m$ of $E_j$. Since $H_{ji}^m\cap X_j\subset E_{j}\cap X_j$, we have that $S^1\cap A=S_A^1=\pi^j(X_j\cap E_j)$ is nonempty as well which is absurd since $A$ does not meet $S^1$ by hypothesis. In particular, all of the blowups $\pi_j,\,\ldots,\,\pi_1$ in the principalization are along empty centers hence they are isomorphisms which can be identified with the identity therefore $X_j=X_0=A\cap X$. Recall that $j$ is the index in Proposition \ref{embedded_des} hence, since $A\cap X$ is irreducible by hypothesis and coincides with $X_j$, we conclude that $C_j=X_j=A\cap X$. Omitting the empty blowups and reindexing the blowup sequence, we conclude the proof of the claim.


Now we construct another subset $S^2\subset X-S^1$ as follows. For every precompact neighbourhood $A$ in $M-S^1$, every irreducible component $C_{i}^j$ of the center $C_i$ and every irreducible component $H_{ik}^m$ of the divisor $E_i$, consider the real analytic proper map
$$\phi_{A,i,j,k,m}\,=\,\pi^{i}|_{C_{i}^j\cap H_{ik}^m}:C_{i}^j\cap H_{ik}^m \rightarrow A\cap X$$
and its set of critical values $S_{A,i,j,k,m}\subset A\cap X$. Again by item \ref{preimage_of_X0} in Lemma \ref{principalization_algorithm_bye}, this map is well defined.

Because $C_{i}^j$ has only simple normal crossings with $H_{ik}^m$, if their intersection is non empty, then either their intersection is transversal or there is a point $p$ such that $C_{i,p}^j\subset H_{ik,p}^m$ and by Theorem \ref{fully_contained}, $C_{i}^j\subset H_{ik}^m$. If the first alternative holds, that is to say their intersection is transversal, then the intersection is a codimension greater than two submanifold hence $S_{A,i,j,k,m}=\mbox{im}(\phi_{A,i,j,k,m})$.

Define the set $S_A^2$ as the union of all the critical sets $S_{A,i,j,k,m}$:
$$S_A^2\,=\,\bigcup_{i,j,k,m}\,S_{A,i,j,k,m}.$$
Because it is a finite union, by Corollary \ref{Sard} it is clear that $S_A^2$ is a closed nowhere dense zero Lebesgue measure subset in $A\cap X$. Because of the functoriality, we have the relation
$$S_{A\cap A'}^2\,=\,S_A^2\cap S_{A'}^2$$
hence, because $M-S^1$ is locally compact and second axiom, these sets define a closed nowhere dense zero Lebesgue measure subset $S^2$ in $X-S^1$ such that $S^2\cap A=S_A^2$ for every precompact neighbourhood $A$ in $M$.

Define the subset $S\,=\,S^1\cup S^2$ in $X$ and consider a precompact neighbourhood $A$ in $M-S$ such that $A\cap X$ is non empty.

Our \textbf{fourth claim} is that, omitting the empty blowups, the principalization on $A$ has the following properties:
\begin{enumerate}
\item $C_0=A\cap X$. In particular, the strict transforms of $A\cap X$ are empty and $C_i\subset E_i$ for every $i\geq 1$.
\item For every $i\geq 1$, if some irreducible component $C_{i}^j$ of the center $C_i$ has non empty intersection with some irreducible component $H_{ik}^m$ of the divisor $E_i$, then $C_{i}^j\subset H_{ik}^m$.
\item Every map $\phi_{A,i,j,k,m}$ with $C_{i}^j\cap H_{ik}^m$ nonempty is a locally trivial fibration.
\end{enumerate}
\begin{proof}[Proof of the fourth claim]
\ \linebreak
\begin{enumerate}
\item The first statement of the item is the third claim. The second statement follows directly from the definition of the strict transform \eqref{strict_transform_definition} since
$$X_1=\overline{\pi_1^{-1}\left((A\cap X)-C_0\right)}$$
is clearly empty hence the following strict transforms are empty as well. Therefore, by item \ref{center_contained_in_XE} in Lemma \ref{principalization_algorithm_bye}, we have that $C_i\subset E_i$ for every $i\geq 1$ and this proves the first item.
\item If $C_{i}^j\cap H_{ik}^m$ is nonempty and $C_{i}^j$ is not contained in $H_{ik}^m$, then as explained before the intersection is transversal and its image by $\phi_{A,i,j,k,m}$ coincides with $S_{A,i,j,k,m}$ which is nonempty and contained in $S$ and $A\cap X$ as well. The last conclusion is absurd since $A$ does not meet $S$ by hipothesis and this proves the second item.
\item If $C_{i}^j\cap H_{ik}^m$ is nonempty, then by the previous item the map $\phi_{A,i,j,k,m}$ has domain the whole irreducible component $C_{i}^j$,
$$\phi_{A,i,j,k,m}:\,C_{i}^j\rightarrow A\cap X.$$
Because $A$ does not meet $S$, in particular it does not meet the critical set $S_{A,i,j,k,m}$ either hence $\phi_{A,i,j,k,m}$ is a regular map. Because $C_{i}^j$ has codimension at least two by the zeroth claim, $A\cap X$ has codimension two by hypothesis and $\phi_{A,i,j,k,m}$ is a regular map, we conclude that $C_{i}^j$ has codimension two as well and $\phi_{A,i,j,k,m}$ is a proper local isomorphism. The item now follows by Ehresmann's Theorem and this concludes the proof of the fourth claim.
\end{enumerate}
\end{proof}

So far we have achieved the second and fourth items in the statement of the Lemma since these are respectively the first two items of the fourth claim.

Consider a point $p$ in $X-S$. There is a precompact neighbourhood $W_p$ of $p$ in $M-S$ such that $W_p\cap X$ is a common trivializing neighbourhood of the mentioned locally trivial fibrations, that is to say every map
$$\pi^{i}|_{C_{i}^j}: C_{i}^j\rightarrow W_p\cap X$$
is a trivial fibration where $C_{i}^j$ denotes the $j$-th irreducible component of $C_i$. Blowing up one connected component of the fibrations at a time and reordering, we get the irreducibility of the new centers and the isomorphisms
$$\pi^{i}|_{C_i}: C_i\rightarrow W_p\cap X.$$
This constitutes the third item in the statement of the Lemma.

Finally, locally the vector field $F$ admits a smooth extension hence taking the neighbourhood $W_p$ small enough, the first item in the statement of the Lemma is achieved. Again, locally every center is the locus of some strong regular sequence hence taking the neighbourhood $W_p$ small enough we get the fifth item in the statement of the Lemma as well and this concludes the proof.
\end{proof}

It is important to remark that in the last paragraph of the proof the functoriality was not used and the penultimate paragraph in the previous proof, that is the blowing up of one irreducible component of the centers at a time, is the first and only time whereat we break the functoriality of the principalization. However, the functoriality is no longer needed nor used in the rest of the proof hence there is no harm in doing this.

Along the rest of the proof, we suppose that $X$ has codimension two and \textbf{we fix the strong regular sequences $\mathcal{R}_i$} given in item \ref{regular_sequences} in Lemma \ref{limpiando_la_cancha}. As it was explained at the end of subsection \ref{connection}, these sequences define an Ehresmann connection on every irreducible component of the exceptional divisor $E_{W_p}$. For notational convenience, \textbf{we may further restrict $M$ to the neighbourhood $W_p$ in Lemma \ref{limpiando_la_cancha} and assume that there is a finite principalization on $M$ satisfying the items in the Lemma,}
\begin{equation}\label{principalization}
\textit{\textsf{prin}}_{\mathcal{I}_U}\,:\,\tilde{M}=M_r\xrightarrow{\pi_r} \ldots\rightarrow M_2\xrightarrow{\pi_2} M_1\xrightarrow{\pi_1} M_0= M.
\end{equation}
Since the centers are irreducible, we suppress the super index in the notation for the irreducible components of the centers and the exceptional divisors.
Recall that $X$ is a smooth codimension two connected and closed real analytic submanifold.

Now suppose that the vector field $F$ on $X$ is tangent to it and define the smooth vector field $F_i$ on the center $C_i$ as the unique field whose pushout by the isomorphism $\pi^{i}_0|_{C_i}$ is the original $F$, that is to say
$$(\pi^{i}_0|_{C_i})_*\,(F_i)=F.$$
By item \ref{maps_isomorphisms} in Lemma \ref{limpiando_la_cancha}, the vector fields $F_i$ are well defined.

Recall that in view of Lemma \ref{limpiando_la_cancha}, since the centers are irreducible the components $H_{ij}$ are irreducible as well and we have suppressed the unnecessary super index. In the following lemma we make the usual identification of the irreducible component $H_{ik}$ of $E_i$ with the corresponding one $H_{kk}$ in $E_k$, that is to say we identify the isomorphism \eqref{isomorphism_irr_components}
$$\pi^i_k|_{H_{ik}}:H_{ik}\rightarrow H_{kk}$$
with the identity map. In particular, if $C_i\subset H_{ik}$, then this identification implies $C_i\subset H_{kk}$ and $F_i=F_k$.

\begin{lema}\label{Induction_Up}
There is a smooth vector field $\tilde{V}$ in $\tilde{M}$ tangent to the exceptional divisor $\tilde{E}$ such that on every irreducible component $H_k$ of $\tilde{E}$, identifying this component with the corresponding one in $E_k$, the field is the sum of the horizontal lifting of $F_{k-1}$ on $C_{k-1}$ and a vertical field.
\end{lema}
\begin{proof}
A smooth vector field $V_l$ in $M_l$ satisfies the property $P(l)$ if:
\begin{enumerate}
\item It is tangent to the center $C_l$ and restricted to it coincides with $F_l$.
\item It is tangent to every irreducible component $H_{lk}$ of the exceptional divisor $E_l$ and restricted to each one of them, identifying this component with the corresponding one $H_{kk}$ in $E_k$, the field is the sum of the horizontal lifting of $F_{k-1}$ on $C_{k-1}$ and a vertical field.
\end{enumerate}

The proof is by induction on the depth parameter of the principalization.
\bigskip

\noindent\textbf{Basis case:} By item \ref{field_extension} in Lemma \ref{limpiando_la_cancha} there is a smooth vector field $\hat{V}$ in $M$ extending the field $F$ in $X$, that is to say $\hat{V}|_X=F$. By item \ref{first_center} in Lemma \ref{limpiando_la_cancha} $C_0=X$ and because there is no exceptional divisor at this stage and $F$ is tangent to $X$ by hypothesis, $V_0=\hat{V}$ satisfies $P(0)$.
\bigskip

\noindent\textbf{Induction step:} Let $V_l$ be a smooth vector field in $M_l$ satisfying $P(l)$. By Lemma \ref{Pullback} there is a smooth vector field $V_l^{*}$, the pullback of $V_l$ by the blowup $\pi_{l+1}$, defined on $M_{l+1}$.
\bigskip

\noindent\textbf{Claim:} \textit{The vector field $V_l^{*}$ satisfies the second item of $P(l+1)$.}

It is clear by Lemma \ref{Pullback} that $V_l^{*}$ satisfies the second item on the new irreducible component $H_{l+1,\,l+1}$ of $E_{l+1}$. On any other irreducible component $H_{l+1,\,j}$, the blowup $\pi_{l+1}$ is an isomorphism over its image,
$$\pi_{l+1}|_{H_{l+1,\,j}}:\,H_{l+1,\,j}\xrightarrow{\cong} H_{lj},\qquad j<l+1,$$
hence $V_l^{*}$ satisfies the second item as well on $H_{l+1,\,j}$ since $V_l$ does so on $H_{lj}$ and this proves the claim.
\bigskip

By item \ref{first_center} in Lemma \ref{limpiando_la_cancha}, there are no non empty strict transforms of $X$ therefore, the center $C_{l+1}$ is contained in the exceptional divisor $E_{l+1}$ and by item \ref{fully_contained_item} in the same lemma, $C_{l+1}$ must be fully contained in some irreducible component of $E_{l+1}$. By dimension and the fact that $C_{l+1}$ has only simple normal crossings with $E_{l+1}$, there are only two possible cases:
\bigskip

\noindent\textbf{Case1:} \textit{The center $C_{l+1}$ belongs to two different irreducible components $H_{l+1,\,i}$ and $H_{l+1,\,j}$ of $E_{l+1}$.}

Because $V_l^{*}$ is tangent to $H_{l+1,\,i}$ and $H_{l+1,\,j}$, by dimension it is also tangent to $C_{l+1}$. Under the usual identification of the irreducible component $H_{l+1,\,i}$ of $E_{l+1}$ with the respective one $H_{ii}$ of $E_{i}$, the field $V_l^{*}$ is the sum of a vertical field and the horizontal lifting of $F_{i-1}$ on $C_{i-1}$. Moreover, $V_l^{*}$ coincides with $F_{l+1}=F_{i}$ on $C_{l+1}$ and the last equality is due to the previous identification. Indeed,
$$(\pi_i|_{C_{l+1}})_*\,(V_l^{*}|_{C_{l+1}})=F_{i-1}$$
and by definition of $F_i$ and $F_{i-1}$, this proves the claim. The same calculation with $H_{l+1,\,j}$ instead of $H_{l+1,\,i}$ gives the same conclusion therefore there is no ambiguity. 

Then, the field $V_{l+1}=V_l^{*}$ satisfies $P(l+1)$.
\bigskip

\noindent\textbf{Case2:} \textit{The center $C_{l+1}$ belongs to only one irreducible component $H_{l+1,\,i}$ of $E_{l+1}$.}

The map $\pi^{l+1}_{i-1}|_{C_{l+1}}:C_{l+1}\rightarrow C_{i-1}$ is an isomorphism for
$$\pi^{l+1}_{i-1}|_{C_{l+1}}\circ \pi^{i-1}_0|_{C_{i-1}}= \pi^{l+1}_0|_{C_{l+1}}$$
and the maps $\pi^{j}_0|_{C_{j}}$ are isomorphisms for every $j$ by item \ref{maps_isomorphisms} in Lemma \ref{limpiando_la_cancha}.

As usual, identifying the irreducible component $H_{l+1,\,i}$ of $E_{l+1}$ with the respective one $H_{ii}$ of $E_{i}$, we have that $C_{l+1}$ is isomorphic to $C_{i-1}$ by $\pi_i|_{C_{l+1}}$. In particular, the center $C_{l+1}$ is the image of a smooth section $\mu$ of the fibration $\pi_i: H_{ii}\rightarrow C_i$. Because $\pi_i\,\mu=id$, we have that $(\pi_i\,\mu)_*$ and $id-(\pi_i\,\mu)_*$ are orthogonal central idempotent operators on the space of vector fields on $C_{l+1}$ giving a unique splitting of the field in its tangent and vertical components respectively. In particular, there is a unique splitting on $C_{l+1}$
$$V_l^{*}= V_l^{*,\,tg}+ V_l^{*,\,vert}$$
in the components tangent to and vertical on $C_{l+1}$ respectively where
$$(\pi_i\,\mu)_* V_l^{*}= V_l^{*,\,tg},$$
$$\left(id-(\pi_i\,\mu)_*\right) V_l^{*}= V_l^{*,\,vert}.$$

By Lemma \ref{extension2}, there is a neighborhood $W$ of $H_{l+1,\,i}=H_{ii}$ in $M_{l+1}$ and a vector field $V_{l+1}^{\,vert}$ defined on $W$ such that restricted to $H_{l+1,\,i}$ is vertical and extends $V_l^{*,\,vert}$.

Consider a smooth bump function $\rho: M_{l+1}\rightarrow [0,1]$ satisfying the following:
\begin{enumerate}
\item Its support is contained in $W$.
\item Its support is disjoint from every other irreducible component in $E_{l+1}$ different from $H_{l+1,\,i}$.
\item It equals one on some neighbourhood of the center $C_{l+1}$.
\end{enumerate}
These conditions are achieved since the center $C_{l+1}$ is a submanifold which belongs only to the irreducible component $H_{l+1,\,i}$ by hypothesis hence it is disjoint from any other irreducible component by item \ref{fully_contained_item} in Lemma \ref{limpiando_la_cancha}.

Define the following smooth vector field, now defined on the whole $M_{l+1}$ extending it by zero outside $W$:
$$Vert_{\,l+1}= -\rho\,V_{l+1}^{\,vert}.$$
It is clear that, restricted to $H_{l+1,\,i}$ the vector field $Vert_{\,l+1}$ is vertical, it is zero on any other irreducible component in $E_{l+1}$ different from $H_{l+1,\,i}$ and
$$V_{l+1}= V_l^{*} + Vert_{\,l+1}$$
is tangent to $C_{l+1}$ for
$$V_{l+1, p}= V_{l, p}^{*} + Vert_{\,l+1, p}= V_{l, p}^{*} - V_{l, p}^{*,\,vert}= V_{l, p}^{*,\,tg}$$
for every $p$ in $C_{l+1}$.

Moreover, $V_{l+1}$ coincides with $F_{l+1}=F_{i}$ on $C_{l+1}$ and the last equality is due to the previous identification. Indeed,
$$(\pi_i|_{C_{l+1}})_*\,(V_{l+1}|_{C_{l+1}})=(\pi_i|_{C_{l+1}})_*\,(V_l^{*}|_{C_{l+1}})=F_{i-1}$$
and by definition of $F_i$ and $F_{i-1}$, this proves the claim.

Then, the field $V_{l+1}$ satisfies $P(l+1)$. This concludes the induction argument.
\bigskip

Finally, because $V_r$ satisfies $P(r)$, defining $\tilde{V}$ as the vector field $V_r$ on $M_r= \tilde{M}$, we have the result.
\end{proof}

\begin{lema}\label{Induction_downstairs}
There is a locally Lipschitz continuous vector field $V$ on $M$ such that it is the pushout of the smooth vector field $\tilde{V}$ in $\tilde{M}$ in Lemma \ref{Induction_Up} by the principalization $\textsf{prin}_{\mathcal{I}_U}$ in expression \eqref{principalization} and coincides with the vector field $F$ on $X$.
\end{lema}
\begin{proof}
We will say that a vector field $V^l$ on $M_l$ satisfies the property $Q(l)$ if it is the pushout of $\tilde{V}$ by $\pi_l^{r}$ and coincides with $F_l$ on $C_l$. By Lemma \ref{Pushout}, a vector field satisfying the property $Q(l)$ is locally Lipschitz continuous. We will prove that there exists a vector field $V^{l}$ satisfying $Q(l)$ by descending induction on $l$.
\bigskip

\noindent\textbf{Basis case:} Put $V^{r}=\tilde{V}$. Since $C_r$ is empty, the statement $Q(r)$ follows immediately.
\bigskip

\noindent\textbf{Induction step:} There is a vector field $V^{l}$ on $M_l$ satisfying $Q(l)$. By Lemma \ref{Induction_Up} and identifying as usual the irreducible component $H_{rl}$ in $E$ with the respective one $H_{ll}$ in $E_l$, we have that on $H_{ll}$ the field $V^{l}$ is the sum of the horizontal lifting of $F_{l-1}$ on $C_{l-1}$ and a vertical field. By Lemma \ref{Pushout}, there is a vector field $V^{l-1}=(V^l)_*$ on $M_{l-1}$ satisfying $Q(l-1)$. This concludes the induction argument.
\bigskip

Finally, defining $V$ as the vector field $V^{0}$ on $M$, we have the result.
\end{proof}

\begin{proof}[Proof of Theorem \ref{Theorem_codim2}]
It remains to show that the vector field $V$ in the previous Lemma \ref{Induction_downstairs} is actually weakly logarithmic for $U$.

Denote by $\tilde{U}$ the pullback of the potential $U$ by the principalization, that is to say $\tilde{U}= U\circ \textit{\textsf{prin}}_{\mathcal{I}_U}$. By item \ref{zero_locus_U_tilde} in Lemma \ref{rmrk-1}, we have that $\tilde{E}=\mathcal{V}(\tilde{U})$ hence there is a smooth real valued function $\tilde{P}$ defined on $\tilde{M}-\tilde{E}$ such that
$$\tilde{V}(\tilde{U})= \tilde{P}\,\tilde{U}.$$

First, we will prove that the function $\tilde{P}$ smoothly extends to $\tilde{E}$. Consider a point $p$ in $\tilde{E}$. By item \ref{monomialization} in Lemma \ref{rmrk-1}, there is a coordinate neighbourhood $W_p$ of $p$ in $\tilde{M}$ with real analytic coordinates $w_1,\ldots, w_n$ such that
$$\tilde{U}|_{W_p}\,  =\, \pm w_1^{d_1}\ldots w_n^{d_n}$$
for some positive integer $d_1$ and some nonnegative integers $d_2,\ldots, d_n$.

Without loss of generality, we may suppose that the first $d_1,\ldots, d_l$ are nonzero while $d_{l+1}=\ldots= d_n=0$. Equivalently, the intersection of $W_p$ with the union of the irreducible components of $\mathcal{V}(\tilde{U})$ containing $p$ is $\mathcal{V}(w_1)\cup\ldots\cup \mathcal{V}(w_l)$.

Because $\tilde{V}$ is smooth and tangent to the irreducible components $\mathcal{V}(w_1),\ldots, \mathcal{V}(w_l)$ on $W_p$, this vector field has the expression
$$\tilde{V}|_{W_p}= h_1\,w_1\ \partial_{w_1}+\ldots + h_l\,w_l\ \partial_{w_l} + h_{l+1}\, \partial_{w_{l+1}} + \ldots + h_n\,\partial_{w_n}$$
for some real valued smooth functions $h_1,\ldots, h_n$. Therefore, $\tilde{P}$ smoothly extends to $W_p$ just defining it over this neighbourhood, with respect to the coordinates of $W_p$, with the expression
$$\tilde{P}|_{W_p}=  d_1\,h_1 + \ldots + d_l\,h_l.$$
Because the point $p$ was arbitrary and the extension of $\tilde{P}$ to $W_p$ is unique, we conclude that $\tilde{P}$ extends smoothly to $\tilde{E}$.

Now, we prove that $V$ is weakly logarithmic for $U$. Since $X=\mathcal{V}(U)$, there is a continuous real valued function $P$ defined on $M-X$ such that
$$V(U)=P\,U.$$
From item \ref{zero_locus_U_tilde} in Lemma \ref{rmrk-1} and
$$(P\circ \textit{\textsf{prin}}_{\mathcal{I}_U})\,\tilde{U}= (P\,U)\circ\textit{\textsf{prin}}_{\mathcal{I}_U}
= V(U)\circ \textit{\textsf{prin}}_{\mathcal{I}_U} = \tilde{V}(\tilde{U})= \tilde{P}\,\tilde{U},$$
it immediately follows that
\begin{equation}\label{P_y_P_tilde}
P\circ \textit{\textsf{prin}}_{\mathcal{I}_U} =\tilde{P}\qquad\mbox{on}\qquad \tilde{M}-\tilde{E}.
\end{equation}

We define $P$ on $X$ just as the zero constant. We claim that $P$ is bounded over compact sets. Indeed, consider a compact set $K\subset M$ and its preimage $\tilde{K}\subset \tilde{M}$ by the map $\textit{\textsf{prin}}_{\mathcal{I}_U}$ which is a compact set as well because the map is proper. Because of the expression \eqref{P_y_P_tilde} we have
$$P(K)=P(K\cap X)\cup P(K-X)\subset \{0\}\cup \tilde{P}(\tilde{K}-\tilde{E})\subset \{0\}\cup \tilde{P}(\tilde{K}),$$
that is to say we have the inclusion
\begin{equation}\label{P_y_P_tilde_2}
P(K)\subset \,\{0\}\cup \tilde{P}(\tilde{K}).
\end{equation}
The right hand side in relation \eqref{P_y_P_tilde_2} is bounded since $\tilde{P}$ is smooth and $\tilde{K}$ is compact hence $P(K)$ is bounded as well and because the compact set $K$ was arbitrary, this proves the claim and we have the result.
\end{proof}


\section{A Lyapunov instability criterion}\label{instability_criterion}

\subsection{Statement of the theorem}

Let $M$ be a smooth manifold and consider the smooth mechanical Lagrangian in $TM$
\begin{equation}\label{Lagrangian}
L(x,v)= Q_{x}(v)- U(x),\qquad (x,v)\in TM
\end{equation}
where $Q=(x\mapsto Q_{x})$ is a smooth section of positive definite quadratic forms and $U$ is a nonnegative smooth potential whose zero locus is non empty. By polarization, the section $Q$ defines a Riemannian metric $\rho$ on $M$ such that
\begin{equation}\label{Kinetic_term}
Q_x(v)=\Vert\,v\,\Vert_x^2/2,\qquad x\in M.
\end{equation}
Every zero potential point is a critical point of the potential hence the set
$$[U=0]\times\{ {\bf 0} \}\,\subset\,TM$$
consists entirely of equilibrium points of the Lagrangian dynamics of \eqref{Lagrangian}. This will be proved in Lemma \ref{zeroes_are_critical}.

In this section we prove the following instability criterion:

\begin{teo}\label{Instability}
Consider a nonnegative smooth potential $U$ on a smooth manifold $M$ and let $p$ be a zero potential point. If there is a locally Lipschitz continuous weakly logarithmic for $U$ vector field in some neighbourhood of $p$ in $M$ and non null at this point, then $(p,{\bf 0})$ is a Lyapunov unstable equilibrium point of the Lagrangian dynamics of \eqref{Lagrangian}.
\end{teo}

The existence of a vector field satisfying the hypothesis of the criterion does not depend on the kinetic term of the Lagrangian \eqref{Lagrangian}. In particular, this remark immediately implies the following corollary:

\begin{cor}
Under the hypothesis of Theorem \ref{Instability}, the potential does not exhibit the Bertotti-Bolotin phenomenon at the point because the equilibrium point is Lyapunov unstable in the Lagrangian dynamics of \eqref{Lagrangian} with an arbitrary kinetic term.
\end{cor}

An equivalent statement of the previous corollary is that, under the hypothesis of Theorem \ref{Instability}, the instability of the equilibrium does not depend on the Riemannian metric on $M$ where the kinetic term is given by \eqref{Kinetic_term}.

Theorem \ref{Instability} clearly generalizes Theorem 1.1 in our earlier paper \cite{BP} where the vector field was assumed to the gradient of some regular function. In particular, following the first example in subsection \ref{weakly_log_subsection} regarding quasi homogeneous smooth functions, Theorem \ref{Instability} also generalizes Corollary 2 in \cite{BP} now in the context of Lagrangian dynamics:

\begin{cor}
Consider a nonnegative quasi homogeneous smooth potential $U$ on $M=\R^n$ such that
$$U(t^{\lambda_1}\,x_1,\ldots,\,t^{\lambda_n}\,x_n)\,=\,t^\lambda\,U(x_1,\ldots,\,x_n),\qquad t>0.$$
Then, every zero potential point $(x_1,\ldots,\,x_n)$ such that $\lambda_i\,x_i\neq 0$ for some index $i$ is a Lyapunov unstable equilibrium point of the Lagrangian dynamics of \eqref{Lagrangian}.
\end{cor}

The criterion for instability is sufficient but not necessary as the following example shows. Consider the potential
$$U(x,y)=x^2\,y^2.$$
The axes are invariant subspaces of the corresponding Newtonian dynamics hence the origin is Lyapunov unstable since a particle can escape to infinity with arbitrary small velocity along any of the axes. However, there is no continuous weakly logarithmic vector field non null at the origin.

\subsection{Preliminaries on approximations of unity}

Consider an open set $\Omega\subset\R^n$ and a real valued locally integrable function $f:\Omega\rightarrow\R$.

\begin{defi}
A $C^1$ approximation of unity is a sequence $(u_\eta)_{\eta>0}$ such that for every $\eta>0$:
\begin{enumerate}
\item $u_\eta$ is a nonnegative $C^1$ real valued function on $\R^n$.
\item $supp(u_\eta)\subset \overline{B({\bf 0}, \eta)}\subset\R^n$.
\item $\int_{\R^n}dz\, u_\eta(z)=1$ where $dz$ denotes the Lebesgue measure.
\end{enumerate}
\end{defi}

Denote by $f_\eta$ the convolution
$$f_\eta=u_\eta* f:\Omega_\eta\rightarrow\R,\quad f_\eta(x)=\int_\Omega\,dy\,u_\eta(x-y)\,f(y)$$
where we have defined
$$\Omega_\eta=\{\,x\in \Omega\ |\ x-supp(u_\eta)\subset \Omega\,\}\subset\Omega.$$
For every $\eta>0$, the function $f_\eta$ is $C^1$ on $\Omega_\eta$ with
$$\partial_a f_\eta= (\partial_a u_\eta)* f.$$

\begin{lema}
If $f$ is Lipschitz with Lipschitz constant equals to $k$, then for every $\eta>0$ and every $a=1,\ldots, n$ we have on $\Omega_\eta$
$$\vert\,\partial_a f_\eta\,\vert\leq k.$$
\end{lema}
\begin{proof}
This follows immediately from
$$\vert f_\eta(x')-f_\eta(x)\vert=\left\vert\int_{supp(u_\eta)}dz\,u_\eta(z)\,(f(x'-z)-f(x-z))\right\vert$$
$$\leq\int_{supp(u_\eta)}dz\,u_\eta(z)\,\vert f(x'-z)-f(x-z)\vert$$
$$\leq \int_{supp(u_\eta)}dz\,u_\eta(z)\,k\,\vert x'-x\vert=k\,\vert x'-x\vert,\quad x,x'\in \Omega_\eta.$$
In effect, taking $x$ in $\Omega_\eta$ and $x'=x +\varepsilon\,e_a$ with $\varepsilon>0$ small enough such that $x'$ belongs to $\Omega_\eta$ as well where $e_a$ is the canonical $a$-th vector we have
$$\vert\,\partial_a f_\eta(x)\,\vert= \lim_{\varepsilon\to 0^+}\,\varepsilon^{-1}\vert f_\eta(x +\varepsilon\,e_a)-f_\eta(x)\vert\leq k$$
and this concludes the proof.
\end{proof}

\subsection{The limit curve}

This subsection follows closely the construction of the limit curve in the previous work of the authors in \cite{BMP}, \cite{BP} and \cite{Bu2}. We include it in here to achieve a self contained exposition.

Denote by $X$ the real analytic variety $\mathcal{V}(U)$. Let $p$ be a point in $X$, $V$ be the weakly logarithmic field in the hypothesis and for every $\varepsilon>0$ consider the solution $x^{\varepsilon}$ of the Euler-Lagrange equations of the Lagrangian \eqref{Lagrangian} with initial conditions $x^{\varepsilon}(0)=p$ and $\dot{x}^{\varepsilon}(0)=\varepsilon\, V(p)$.

For every $\varepsilon>0$, define $x_\varepsilon$ such that $x_\varepsilon(\tau)= x^{\varepsilon}(\tau/\varepsilon)$ where $x^{\varepsilon}$ is defined. These are solutions of the Euler-Lagrange equations of the Lagrangian:
\begin{equation}\label{Lagrangian2}
L_\varepsilon(x,v)= Q_{x}(v)-\varepsilon^{-2} U(x),\qquad (x,v)\in TM.
\end{equation}
Now, the initial conditions $x_{\varepsilon}(0)=p$ and $\dot{x}_{\varepsilon}(0)= V(p)$ are fixed but the motion equations become singular as $\varepsilon\to 0^{+}$. Denote by $I_\varepsilon$ the maximal interval containing zero where $x_{\varepsilon}$ is defined.

\begin{lema}\label{Lema1}
For every $\varepsilon>0$, $\Vert\dot{x}_\varepsilon(\tau)\Vert\leq \Vert V(p) \Vert$ for every $\tau$ in $I_\varepsilon$ and
$${\rm Im}(x_\varepsilon)\subset [U\leq \varepsilon^{2}\,\Vert V(p) \Vert^{2}/2].$$
\end{lema}
\dem
For every $\varepsilon>0$, the Hamiltonian
$$H_\varepsilon(x,v)= \Vert v \Vert^{2}/2+\varepsilon^{-2}U(x)$$
is constant along the solution $x_\varepsilon$ hence
$$\Vert \dot{x}_\varepsilon(\tau) \Vert^{2}/2,\ \varepsilon^{-2}U(x_\varepsilon(\tau))\leq H_\varepsilon(x_\varepsilon(\tau), \dot{x}_\varepsilon(\tau))= H_\varepsilon(p, V(p))= \Vert V(p) \Vert^{2}/2$$
and the result follows.
\fdem

\begin{cor}\label{Cor1}
Let $T>0$. For every $\varepsilon>0$ and every $\tau$ in $I_\varepsilon\cap [-T, T]$, $(x_{\varepsilon}(\tau), \dot{x}_\varepsilon(\tau))\in R_T$ where
$$R_T= \left\lbrace\,(x,v)\in TM\ |\ x\in \overline{B_\rho(p, T\, \Vert V(p) \Vert)},\ v\in  \overline{B_{x}({\bf 0}, \Vert V(p) \Vert)}\,\right\rbrace.$$
Note that this region is a compact set not depending on $\varepsilon$.
\end{cor}
\dem
By Lemma \ref{Lema1}, $\Vert \dot{x}_\varepsilon(\tau)\Vert\leq \Vert V(p) \Vert$ and
$$d( x_\varepsilon(\tau),\,p) \leq \left\vert\int_0^{\tau}ds\ \Vert\dot{x}_\varepsilon(s)\Vert\ \right\vert\leq |\tau|\ \Vert V(p) \Vert\leq T \Vert V(p) \Vert,$$
the result follows.
\fdem

\begin{cor}\label{Cor2}
For every $\varepsilon>0$, $x_\varepsilon$ is defined over the whole real line.
\end{cor}
\dem
Consider the maximal interval $I_\varepsilon=(\omega_-, \omega_+)$ and suppose that $\omega_+$ is finite. Then, $(x_\varepsilon, \dot{x}_\varepsilon)|_{[0,\omega_+)}$ is contained in the compact set $R_{\omega_+}$ which is absurd hence $\omega_+=+\infty$. Analogously, $\omega_-=-\infty$.
\fdem
\bigskip

\textbf{Denote by $E$ the neighbourhood of $p$ in $M$ where the vector field $V$ is defined. Let $T>0$ be small enough such that $\overline{B_\rho(p, T\, \Vert V(p) \Vert)}$ is contained in $E$.}
\bigskip

In particular, for every $\varepsilon>0$ the image of $x_\varepsilon$ restricted to $[-T,T]$ is contained in $E$.

\begin{cor}\label{Cor3}
There is a continuous curve $x:[-T,T]\rightarrow X$ with $x(0)=p$ and a sequence $(\varepsilon_n)$ such that $\varepsilon_n>0$, $\varepsilon_n\to 0^{+}$ and $x_{\varepsilon_n}\rightarrow x$ uniformly on $[-T,T]$.
\end{cor}
\dem
Consider the family $\mathcal F$ of functions $x_{\varepsilon}$ defined in $[-T, T]$ that are solutions of the Euler-Lagrange equations of  \eqref{Lagrangian2} such that 
$x_{\varepsilon}(0)=p$ and $\dot x_{\varepsilon}(0)= V(p)$.  By Lemma \ref{Lema1} the family $\mathcal F$ is equicontinuous.  
Consider a sequence    $(x_{\varepsilon_n})$ such that $\varepsilon_n\to 0^{+}$. By Arzel\`a--Ascoli Theorem there is a subsequence, that we still call by   $(x_{\varepsilon_n})$, 
converging uniformly to a continuous curve $x$. Because $x_\varepsilon(0)=p$ for every $\varepsilon>0$ we conclude that $x(0)=p$. 
Let  $t\in [-T, T]$. By Lemma \ref{Lema1} 

$$0\leq U(x_{\varepsilon_n}(t))\leq \varepsilon_n^{2}\,\Vert V(p) \Vert^{2}/2.$$

Therefore $U(x(t))=0$ for every $t\in [-T, T]$, i.e.,  

$${\rm Im}(x)\subset   \mathcal{V}(U)=X. $$
and this finishes the proof.
\fdem

The limit curve is the uniform limit of a sequence of curves with bounded velocity and unbounded acceleration, that is to say a \textit{high frequency limit} and it will not be differentiable in general. However, there are situations where the limit curve is differentiable. See the work by the first author in \cite{Bu2} as an example of this phenomenon.

\subsection{The holonomic curve}

For every natural $n$, define the function $y_n: [-T, T]\rightarrow \R$ such that
$$y_n(t)= \int_0^{t}\,d\tau\,\left\langle\, V\left( x_{\varepsilon_n}(\tau)\right),\,\dot{x}_{\varepsilon_n}(\tau)\,\right\rangle_{x_{\varepsilon_n}(\tau)}.$$

\begin{lema}\label{Lema2}
Every $y_n$ is Lipschitz with Lipschitz constant independent of $n$. In particular, the sequence $(y_n)_{n\in\N}$ is equicontinuous and uniformly bounded.
\end{lema}
\begin{proof}
It follows immediately from the Cauchy-Schwarz inequality, Lemma \ref{Lema1} and the continuity of $V$ for
$$|\dot{y}_n(t)|= |\left\langle\, V\left( x_{\varepsilon_n}(t)\right),\,\dot{x}_{\varepsilon_n}(t)\,\right\rangle|
\leq ||V\left( x_{\varepsilon_n}(t)\right)||\,||\dot{x}_{\varepsilon_n}(t)||
\leq M^{2}$$
where $M$ is the maximum of the norm of $V$ over the compact set $\overline{B_\rho(p, T\, \Vert V(p) \Vert)}$ which is clearly independent of $n$ and we have the result.
\end{proof}

\begin{lema}\label{Lema3}
Every $\dot{y}_n$ is Lipschitz with Lipschitz constant independent of $n$. In particular, the sequence $(\dot{y}_n)_{n\in\N}$ is equicontinuous and uniformly bounded.
\end{lema}
\begin{proof}
Let $n$ be a natural number. Because $[-T, T]$ is compact, it is enough to prove that for every $\tau$ in this interval there is $\delta>0$ such that $\dot{y}_n$ restricted to $[\tau-\delta, \tau+ \delta]$ is Lipschitz with Lipschitz constant independent of $n$.

Let $\tau$ be in $[-T, T]$ and $(W,\psi)$ be a coordinate neighbourhood of $x_{\varepsilon_n}(\tau)$ in $M$ small enough such that $W\subset E$ and $V|_W$ is Lipschitz. There is $\delta>0$ such that the segment of the curve $x_{\varepsilon_n}$ on $[\tau-\delta, \tau+ \delta]$ is fully contained in $W$.

With respect to these coordinates, the metric $\rho$, the field $V$ and the velocity read as follows:
$$(\psi^{-1})^{*}(\rho) = g_{ab}\,dw^{a}\otimes dw^{b},\quad V_*= \psi_*(V)= V^{a}\,\partial_{a},\quad \psi_*(\dot{x}_{\varepsilon_n})= \dot{x}_{\varepsilon_n}^{a}\,\partial_{a}$$
where $(x_{\varepsilon_n}^{a})= \psi(x_{\varepsilon_n})$ are the coordinates of the curve and $w_1,\ldots, w_n$ are the coordinates of $(W,\psi)$. The coordinates of the curve verify the Euler-Lagrange equations
\begin{equation}\label{Euler-Lagrange}
\ddot{x}_{\varepsilon_n}^{a} + \Gamma_{ij}^{a}\, \dot{x}^{i}_{\varepsilon_n}\, \dot{x}^{j}_{\varepsilon_n}
+ \varepsilon_n^{-2}\,g^{ab}\, \partial_b \tilde{U}=0
\end{equation}
where $\Gamma_{ij}^{a}$ are the Christoffel symbols of the second kind and $\tilde{U}= U\circ\psi^{-1}$. Note that $V_*$ is weakly logarithmic for $\tilde{U}$:
$$V_*(\tilde{U})= V(U)\circ \psi^{-1}= (P\,U)\circ\psi^{-1}= \tilde{P}\,\tilde{U}$$
where we have defined $\tilde{P}= P\circ\psi^{-1}$ which is clearly bounded over compact sets since $\psi$ is a diffeomorphism over its image.

In particular, on the interval $[\tau-\delta, \tau+ \delta]$ we have
$$\dot{y}_n= g_{ab}\,V^{a}\, \dot{x}^{b}_{\varepsilon_n}$$
where the first two factors are evaluated over the coordinates of the curve. We will show that this function is Lipschitz.

Consider a $C^{1}$ approximation of unity $(u_\eta)_{\eta >0}$ and suppose without loss of generality that $\psi(W)_\eta$ contains the segment of the curve $\psi(x_{\varepsilon_n})$ on $[\tau-\delta, \tau+ \delta]$ for every $\eta>0$. Define $h_\eta: [\tau-\delta, \tau+ \delta]\rightarrow \R$ by the expression 
$$h_\eta= g_{ab}\,V^{a}_\eta\, \dot{x}^{b}_{\varepsilon_n},\qquad V_\eta^{a}= u_\eta * V^{a}, \quad \eta>0$$
where again the first two factors of the left hand expression are evaluated over the coordinates of the curve. The functions $h_\eta$ and $V_\eta^{a}$ uniformly converge to $\dot{y}_n$ and $V^{a}$ respectively on $[\tau-\delta, \tau+ \delta]$ as $\eta\to 0^{+}$. If $V^{a}$ has Lipschitz constant $k$, then
\begin{equation}\label{Bound}
|\,\partial_b\,V_\eta^{a}\,|\leq k, \qquad \eta>0.
\end{equation}

For every $\eta>0$, the function $h_\eta$ is $C^{1}$ and its derivative reads as follows
$$\dot{h}_\eta= g_{ab,c}\,V^{a}_\eta\, \dot{x}^{b}_{\varepsilon_n}\, \dot{x}^{c}_{\varepsilon_n}
+ g_{ab}\,\partial_c V^{a}_\eta\, \dot{x}^{b}_{\varepsilon_n}\, \dot{x}^{c}_{\varepsilon_n}
+ g_{ab}\,V^{a}_\eta\, \ddot{x}^{b}_{\varepsilon_n}$$

By Lemma \ref{Lema1} and relation \eqref{Bound}, the first two terms are bounded over $[\tau-\delta, \tau+ \delta]$ by a constant $k'$ independent of $n$ and $\eta>0$ hence after integrating we have
$$\left| h_\eta(t_2)-h_\eta(t_1)-\int_{t_1}^{t_2}\,g_{ab}\,V^{a}_\eta\, \ddot{x}^{b}_{\varepsilon_n}\right|\leq k'\, |t_2-t_1|.$$
Taking the limit as $\eta\to 0^{+}$ in the expression above gives
\begin{equation}\label{triangle1}
\left| \dot{y}_n(t_2)-\dot{y}_n(t_1)-\int_{t_1}^{t_2}\,g_{ab}\,V^{a}\, \ddot{x}^{b}_{\varepsilon_n}\right|\leq k'\, |t_2-t_1|.
\end{equation}
The integrand of the integral in the expression above is bounded by a constant independent of $n$. In effect, by the Euler-Lagrange equations \eqref{Euler-Lagrange} we have the expression
\begin{eqnarray*}
g_{ab}\,V^{a}\, \ddot{x}^{b}_{\varepsilon_n}&=& -[ij,a]\, \dot{x}^{i}_{\varepsilon_n}\, \dot{x}^{j}_{\varepsilon_n}\, V^{a}
- \varepsilon_n^{-2}\,V^{a}\, \partial_a \tilde{U} \\
&=& -[ij,a]\, \dot{x}^{i}_{\varepsilon_n}\, \dot{x}^{j}_{\varepsilon_n}\, V^{a}
- \varepsilon_n^{-2}\,\tilde{P}\,\tilde{U}
\end{eqnarray*}
and the right hand side is bounded since every coefficient is so on $\psi(W)$, $V_*$ is weakly logarithmic for $\tilde{U}$ and the potential evaluated over curve times $\varepsilon_n^{-2}$ is bounded by a constant independent of $n$ by Lemma \ref{Lema1} again. Therefore, there is a constant $k''$ independent of $n$ such that
\begin{equation}\label{triangle2}
\left| \int_{t_1}^{t_2}\,g_{ab}\,V^{a}\, \ddot{x}^{b}_{\varepsilon_n}\right|\leq 
\int_{t_1}^{t_2}\,\left| g_{ab}\,V^{a}\, \ddot{x}^{b}_{\varepsilon_n}\right|\leq k''\,|t_2-t_1|.
\end{equation}
By triangle inequality and expressions \eqref{triangle1} and \eqref{triangle2} we finally get 
$$\left| \dot{y}_n(t_2)-\dot{y}_n(t_1) \right|\leq k'''\, |t_2-t_1|$$
for every $t_2$ and $t_1$ in $[\tau-\delta, \tau+ \delta]$ where we have defined $k'''= k'+k''$ which is clearly independent of $n$ and this finishes the proof.
\end{proof}

\begin{lema}\label{Lema4}
There is a convergent subsequence of $(y_n)$ in the space $C^{1}[-T, T]$ whose limit $y$ satisfies $y(0)=0$ and $\dot{y}(0)= ||V(p)||^{2}$.
\end{lema}
\begin{proof}
Because $y_n(0)=0$ for every $n$ and by Lemma \ref{Lema2} the sequence $(y_n)$ is equicontinuous, by Arzel\`a-Ascoli Theorem there is a convergent subsequence of $(y_n)$ in the space $C^{0}[-T, T]$ uniformly converging to some continuous function $y$ which clearly satisfies $y(0)=0$

Because $\dot{y}_n(0)=||V(p)||^{2}$ for every $n$ and by Lemma \ref{Lema3} the sequence $(\dot{y}_n)$ is equicontinuous, by Arzel\`a-Ascoli Theorem there is a further convergent subsequence of $(\dot{y}_n)$ in the space $C^{0}[-T, T]$ uniformly converging to some continuous function $e$ which clearly satisfies $e(0)=||V(p)||^{2}$.

It rest to show that $y$ is in $C^{1}$ and $\dot{y}=e$ because in this case the subsequence converges to $y$ in $C^{1}[-T, T]$. In this respect we have
$$y_n(t)= \int_0^{t}\,d\tau\, \dot{y}_n(\tau)$$
and taking the limit as $n\to+\infty$ we get
$$y(t)= \int_0^{t}\,d\tau\, e(\tau)$$
hence $\dot{y}=e$ and in particular $\dot{y}(0)= e(0)= ||V(p)||^{2}$. This concludes the proof.
\end{proof}

\subsection{Proof of Theorem \ref{Instability}}

\begin{prop}\label{prop_instability}
The curve $x$ is not constant neither on $[0,T]$ nor on $[-T, 0]$.
\end{prop}
\begin{proof}
Suppose that the curve $x$ is constant on $[0,T]$, that is to say $x(t)=p$ for every $t$ in $[0, T]$. The proof for the interval $[-T,0]$ is almost verbatim. Along the proof, unless otherwise stated all the curves will be restricted to the interval $[0,T]$.

Take the convergent subsequence $(y_n)$ in Lemma \ref{Lema4} with limit the function $y$. Let $(W,\psi)$ be a coordinate neighbourhood of $p$ in $M$ centered at $p$ and without loss of generality we may suppose that the image of the curves $x_{\varepsilon_n}$ are fully contained in a compact subset $K\subset W$. 

With respect to these coordinates, the metric $\rho$, the field $V|_W$ and the curve $x_{\varepsilon_n}$ read as follows:
$$\rho_*=(\psi^{-1})^{*}(\rho) = g_{ab}\,dw^{a}\otimes dw^{b},\qquad V_*= \psi_*(V)= V^{a}\,\partial_{a},\qquad z_n= \psi\circ x_{\varepsilon_n}$$
where $w_1,\ldots, w_n$ are the coordinates of $(W,\psi)$. Define the vector field $A$ on $\psi(W)$ as the Riesz dual of the one form $\langle\,V_*\,,\,\cdot\,\rangle_{\rho_*}$ with respect to the Euclidean metric. Explicitly,
$$A= \sum_{a=1}^{n}\,A_a\,\partial_a,\qquad A_a= g_{ab}\, V^{b}.$$
Then we have the following equivalent expression for the function $y_n$ but now, instead of using explicitly the metric $\rho$, we use the Euclidean metric:
$$y_n(t)= \int_0^{t}\,d\tau\,\left\langle\, A\left( z_n(\tau)\right),\,\dot{z}_n(\tau)\,\right\rangle_E.$$
The advantage of this alternative formulation is that it enables the following calculation
$$y_n(t)= \int_0^{t}\,d\tau\,\left\langle\, A({\bf 0}),\,\dot{z}_n(\tau)\,\right\rangle_E
+ \int_0^{t}\,d\tau\,\left\langle\, A\left( z_n(\tau)\right)- A({\bf 0}),\,\dot{z}_n(\tau)\,\right\rangle_E$$
$$= \left\langle\, A({\bf 0}),\,\int_0^{t}\,d\tau\,\dot{z}_n(\tau)\,\right\rangle_E
+ \int_0^{t}\,d\tau\,\left\langle\, A\left( z_n(\tau)\right)- A({\bf 0}),\,\dot{z}_n(\tau)\,\right\rangle_E$$
\begin{equation}\label{Truco}
= \left\langle\, A({\bf 0}),\,z_n(t)\,\right\rangle_E
+ \int_0^{t}\,d\tau\,\left\langle\, A\left( z_n(\tau)\right)- A({\bf 0}),\,\dot{z}_n(\tau)\,\right\rangle_E
\end{equation}
because $x_{\varepsilon_n}(0)=p$ and the coordinate neighbourhood is centered at $p$ hence $z_n(0)={\bf 0}$ for every natural $n$.

Consider the quadratic form $(v\mapsto \rho_*(v,v))$ on the tangent bundle $T\,\psi(W)$. Restricted to the unit sphere bundle $\pi: S\,\psi(W)\rightarrow \psi(W)$ it is a strictly positive continuous function and attains a minimum $m>0$ over the compact set $\pi^{-1}(K)$. In particular, by Lemma \ref{Lema1} we have
$$m\,\Vert \dot{z}_n(t) \Vert_E^{2}\leq \Vert \dot{x}_{\varepsilon_n}(t) \Vert_\rho^{2}\leq \Vert V(p) \Vert_\rho^{2}$$
and we have proved the relation
\begin{equation}\label{Truco2}
\Vert \dot{z}_n(t) \Vert_E \leq m^{-1/2}\,\Vert V(p) \Vert_\rho.
\end{equation}

By Cauchy-Schwartz inequality, expression \eqref{Truco} and the previous expression \eqref{Truco2} we have
$$\left|\,y_n(t) \,\right| \leq 
\Vert\, A({\bf 0})\,\Vert_E\, \Vert\,z_n(t)\,\Vert_E
+ m^{-1/2}\,\Vert V(p)\Vert_\rho\,\int_0^{t}\,d\tau\,\Vert\, A\left( z_n(\tau)\right)- A({\bf 0})\,\Vert_E.$$
Because the sequence $(z_n)$ converges uniformly to ${\bf 0}$ as $n\to +\infty$, we have that the sequence $(y_n)$ converges uniformly to the constant zero in the same limit. In particular, the limit $y$ equals the constant zero on $[0,T]$ which is absurd because $\Vert V(p)\Vert$ is not zero hence by Lemma \ref{Lema4} the function $y$ cannot be constant on this interval.

We have proved that the curve $x$ cannot be constant $[0,T]$ and this finishes the proof.
\end{proof}

\begin{proof}[Proof of Theorem \ref{Instability}]
By Proposition \ref{prop_instability}, there is a positive $t_*>0$ such that $x(t_*)$ isdistinct from $x(0)=p$. Define $d= dist_\rho(x(t_*), p)$. Then, there is a natural $n_0$ such that
$$dist_\rho(x(t_*), x_{\varepsilon_n}(t_*))<d/2,\qquad n\geq n_0$$
hence
$$x^{\varepsilon_n}(t_*/\varepsilon_n) = x_{\varepsilon_n}(t_*)\notin B_\rho(p,d/2),\qquad n\geq n_0.$$
However,
$$\left(\,x^{\varepsilon_n}(0)\,,\, \dot{x}^{\varepsilon_n}(0)\,\right)= \left(\,p\,,\, \varepsilon_n\,V(p)\,\right)\xrightarrow{\ \ n\to +\infty\ \ } (\,p\,,\,{\bf 0}\,)$$
therefore we conclude that the equilibrium point $(\,p\,,\,{\bf 0}\,)$ is Lyapunov unstable and we have the result.
\end{proof}


\section{Proof of Theorem \ref{main}}

Let $M$ be a real analytic manifold and consider the smooth mechanical Lagrangian in $TM$
\begin{equation}\label{Lagrangian3}
L(x,v)= Q_{x}(v)- U(x),\qquad (x,v)\in TM
\end{equation}
where $Q=(x\mapsto Q_{x})$ is a smooth section of positive definite quadratic forms and $U$ is a real analytic potential. Every critical point of the potential is an equilibrium point of the respective Lagrangian dynamics.

\begin{lema}\label{zeroes_are_critical}
Consider a nonnegative smooth potential $U$. Then, every zero potential point is critical.
\end{lema}
\begin{proof}
Consider a zero potential point $x$ and suppose that it is a regular point of the potential. By the implicit function theorem, there is a coordinate neighbourhood of $x$ centered at the origin such that the potential is simply the projection onto one of the coordinates. In particular, there are points whereat the potential is negative which is absurd by hypothesis.
\end{proof}

\begin{lema}\label{Lema_semifinal}
Consider a nonnegative real analytic potential $U$ on a three dimensional real analytic manifold $M$. Let $X$ be the set of limit points of the zero locus of $U$, that is to say its derived set. Then, there is an open dense subset of $X$ such that every one of its points is a Lyapunov unstable equilibrium point of the Lagrangian dynamics of \eqref{Lagrangian3}.
\end{lema}
\begin{proof}
By the previous Lemma \ref{zeroes_are_critical}, every point in $X$ is an equilibrium point and it rest to show the existence of an open dense subset of $X$ of Lyapunov unstable points.

By Theorem \ref{smooth_points}, the subset $\mathring{\mathcal{V}(U)}$ of smooth points in the zero locus $\mathcal{V}(U)$ is open and dense in the locus and splits in $d$-dimensional real analytic submanifolds $\mathring{X}^{(d)}$,
$$\mathring{\mathcal{V}(U)}\,=\,\mathring{X}^{(0)}\sqcup\mathring{X}^{(1)}\sqcup\mathring{X}^{(2)}\sqcup \mathring{X}^{(3)}.$$

The set of isolated points of the zero locus $\mathcal{V}(U)$ is $\mathring{X}^{(0)}$. These points are strict minimums of the potential hence by the Lagrange\,-\,Dirichlet Theorem they are Lyapunov stable. In particular,
$$X=\mathcal{V}(U)-\mathring{X}^{(0)}$$
hence the subset $\mathring{X}$ of smooth points of $X$ is open and dense and splits in $d$-dimensional real analytic submanifolds $\mathring{X}^{(d)}$,
$$\mathring{X}\,=\,\mathring{X}^{(1)}\sqcup\mathring{X}^{(2)}\sqcup \mathring{X}^{(3)}.$$

Now we prove that every regular component of $X$, that is to say every connected component of $\mathring{X}$, has an open and dense set of Lyapunov unstable points. We consider each dimension separately.
\bigskip

\noindent\textbf{Dimension one:} If $\mathring{X}^{(1)}$ is empty, then there is nothing to prove. Otherwise, consider a non empty connected component $X_1$ of $\mathring{X}^{(1)}$. Then, $X_1$ is a one dimensional submanifold of $M$ hence there is a non null vector field $F_1$ defined on $X_1$ and tangent to it. By Theorem \ref{Theorem_codim2} there is an open and dense subset $A_1$ of $X_1$ such that every point in $A_1$ has a neighbourhood in $M$ whereat $F_1$ admits a locally Lipschitz continuous weakly logarithmic for $U$ extension. Therefore, by Theorem \ref{Instability} every point in $A_1$ is Lyapunov unstable.
\bigskip

\noindent\textbf{Dimension two:} If $\mathring{X}^{(2)}$ is empty, then there is nothing to prove. Otherwise, consider a non empty connected component $X_2$ of $\mathring{X}^{(2)}$. By Theorem \ref{Theorem_codim1}, there is an open subset $A\subset M$ and a closed nowhere dense zero Lebesgue measure subset $S\subset X_2$ such that
\begin{enumerate}
\item $A\cap \mathcal{V}(U)=X_2$.
\item Every smooth vector field defined on some open subset $B\subset A$ and tangent to $X_2$ at $B\cap X_2$, possibly empty, is smoothly logarithmic for $U$ in $B-S$.
\end{enumerate}

For every point $p$ in $X_2-S$ there is a neighbourhood $W_p$ of $p$ in $A$ and a non vanishing smooth vector field $V$ defined on $W_p$ and tangent to $X_2$ at $W_p\cap X_2$. Indeed, there is a coordinate neighbourhood $(W_p,\,\varphi_p)$ of $p$ in $A$ such that
$$\varphi(W_p\cap X_2)\,=\,\varphi(W_p)\cap [x_3\,=\,0]$$
where $x_1,\,x_2,\,x_3$ are the canonical coordinates of $\R^3$. Define the vector field $V$ on $W_p$ as the pushout of the first canonical vector field by $\varphi_p^{-1}$, that is to say
$$V\,=\,(\varphi_p^{-1})_*({\bf e_1}).$$

Therefore, by Theorem \ref{Instability} every point in the open and dense subset $A_2=X_2-S$ of $X_2$ is Lyapunov unstable.
\bigskip

\noindent\textbf{Dimension three:} If $\mathring{X}^{(3)}$ is empty, then there is nothing to prove. Otherwise, consider a non empty connected component $X_3$ of $\mathring{X}^{(3)}$. By Theorem \ref{analytic_continuation}, the analytic continuation principle, $U$ is identically zero on $X_3$ hence every point in this component is Lyapunov unstable since the non constant motions are reparameterizations of geodesics with respect to the metric induced by the kinetic term through polarization.
\bigskip

We have proved the existence of an open and dense subset in $\mathring{X}$ of Lyapunov unstable points which is open and dense in $X$ as well. This concludes the proof.
\end{proof}

\begin{proof}[Proof of Theorem \ref{main}]
By definition, a point $p$ such that $U(p)=c$ is a minimum of $U$ if and only if
$$p\,\in\,\mbox{Int}\left([\,U\geq c\,]\right)\cap [\,U=c\,]\,=\,Y_c\cap [\,U=c\,]\,=\,Z_c$$
where we have denoted by $Y_c$ the interior of $[\,U\geq c\,]\subset M$. Moreover, the isolated points of $Z_c$ are the the strict minimums points with potential $c$ while the limit points of $Z_c$ are the non strict minimums with potential $c$, that is to say
$$Z_c'\,=\,\left\lbrace\,p\ \mbox{is a non strict minimum of}\ U\ |\ U(p)=c\,\right\rbrace$$
where $Z_c'$ denotes the derived set of $Z_c$, the set of limit points of $Z_c$.

Applying the previous Lemma \ref{Lema_semifinal} to the potential $U-c$ on the three dimensional real analytic manifold $Y_c$, we have that there is an open and dense subset $A_c\subset Z_c'$ whose points are Lyapunov unstable equilibrium points of the Lagrangian dynamics of \eqref{Lagrangian3} because the dynamics of \eqref{Lagrangian3} with $U$ is the same as the one with $U-c$.

If $Z_c'$ is non empty, then applying Lemma \ref{zeroes_are_critical} to $U-c$ we have that $c$ is a critical value of $U$ hence
$$\bigsqcup_{c\,\in\, U(\Sigma\, U)}\, Z_c'\,=\,\left\lbrace\,p\ \mbox{is a non strict minimum of}\ U\,\right\rbrace.$$

Finally, it is clear then that the subset
$$A\,=\,\bigsqcup_{c\,\in\, U(\Sigma\, U)}\, A_c\,\subset\,\bigsqcup_{c\,\in\, U(\Sigma\, U)}\, Z_c'$$
is open and dense in the set of non strict minimums of $U$ and every one of its points is a Lyapunov unstable equilibrium of the Lagrangian dynamics of \eqref{Lagrangian3}. This concludes the proof.
\end{proof}

\section{Acknowledgments}
The authors would like to express a very special thank to Mark Spivakovsky for several suggestions and explanations.


\begin{thebibliography}{coh93}


\bibitem[Ar]{Arnold}
V. I. Arnold, \emph{Arnold's Problems}, Springer-Verlag, (2002).






\bibitem[Bu]{Bu2}
J. M. Burgos, \emph{Strong degenerate constraining in Lagrangian dynamics}, arXiv: 2104.06549.


\bibitem[Br]{Br}
M. Brunella, \emph{Instability of equilibria in dimension three}, Ann I Fourier, {\bf 48} (1998).

\bibitem[BB]{BB}
M. L. Bertotti, S. V. Bolotin, \emph{On the influence of the kinetic energy on the stability of equilibria of natural Lagrangian systems}, Arch. Rational Mech. Anal., {\bf 152} (2000), 65--79.

\bibitem[BM]{Bierstone_Milman}
E. Bierstone, P. D. Milman, \emph{Canonical desingularization in characteristic zero by blowing up the maximum strata of a local invariant},  Invent. math., {\bf 128} (1997), 207--302.


\bibitem[BP]{BP}
J. M. Burgos, M. Paternain, \emph{On the Lyapunov instability in Lagrangian dynamics}, Proc. Amer. Math. Soc., (2022)

\bibitem[BGT]{BGT}
R. B. Bortolatto, M. V. P. Garcia, F. A. Tal, \emph{Kinetic Energy and the Stable Set}, Qualitative Theory of Dynamical Systems, {\bf 10} (2011). 

\bibitem[BMP]{BMP}
J. M. Burgos, E. Maderna, M. Paternain, \emph{On the Lyapunov instability in Newtonian dynamics}, Nonlinearity, {\bf 34} (2021), 6719--6726. 

\bibitem[Ca1]{Cano1}
F. Cano, \emph{Desingularization strategies for three-dimensional vector fields}, Springer Lecture Notes {\bf 1259} (1987).

\bibitem[Ca2]{Cano2}
F. Cano, \emph{Final forms for a three dimensional vector field under blowing-up}, Ann. Inst. Fourier, {\bf 37} no. 2 (1987), 151--193.

\bibitem[Car]{Cartan}
H. Cartan, \emph{Vari\'et\'es analytiques r\'eelles et vari\'et\'es analytiques complexes}, Bulletin de la Soc. Math. de France, {\bf 85} (1957), 77--99.

\bibitem[Ch]{Cetaev}
N.G. Chetaev, \emph{On the instability of equilibrium in certain cases when the force function is not a maximum}, Akad. Nauk SSSR Prikl. Math. Mekh.,  {\bf 16} (1952), 89--93.

\bibitem[Di]{Dirichlet}
L. G. Dirichlet, \emph{\"Uber die Stabilitat des Gleichgewichts}, J. Reine Angew. Math. {\bf 32} (1846) 85--88.



\bibitem[GT]{GT2}
M. V. P. Garcia, F. A. Tal, \emph{The influence of the kinetic energy in equilibrium of Hamiltonian systems}, J. of Differential Equations, {\bf 213} (2005), 410--417. 



\bibitem[Ha]{Ha}
P. Hagedorn, \emph{Die Umkehrung der Stabilit\"atss\"atze von Lagrange\,-\,Dirichlet und Routh}, Arch. Rational Mech. Anal., {\bf 42} (1971), 281--316.


\bibitem[Hart]{Hart}
R. Hartshorne, \emph{Algebraic Geometry}, Graduate texts in Mathematics, Springer (1977).

\bibitem[Hi]{Hironaka}
H. Hironaka, \emph{Resolution of singularities of an algebraic variety over a field of characteristic
zero}, Annals of Math., {\bf 79} (1964), 109--326.

\bibitem[HM]{Hagedorn2}
P. B. Hagedorn, J. Mawhin, \emph{A simple variational approach to a converse of the Lagrange\,-\,Dirichlet theorem}, Arch. Rational Mech. Anal., {\bf 120} (1992), 327--335.

\bibitem[Ko]{Ko}
V. V. Kozlov, \emph{Instability of equilibrium in a potential field}, Russian Math. Surveys, {\bf 36} (1981), 238--239.

\bibitem[Ko2]{Ref2}
V. V. Kozlov, \emph{Asymptotic solutions of equations of classical mechanics}, J. Appl. Math. Mech., \textbf{46} (1982), 454--7.

\bibitem[Ko3]{Ref3}
V. V. Kozlov, \emph{Asymptotic motions and the inversion of the Lagrange\,-\,Dirichlet theorem}, J. Appl. Math. Mech., \textbf{50} (1987), 719--25.



\bibitem[Ku]{Ref4}
A. N. Kuznetsov, \emph{On existence of asymptotic solutions to a singular point of an autonomic system possessing a formal solution}, Functional Anal. Appl.,  \textbf{23} (1989).

\bibitem[KN]{Kobayashi_Nomizu}
S. Kobayashi, K. Nomizu, \emph{Foundations of Differential Geometry}, Wiley Classics Library, Vol. 1, (1963).

\bibitem[KP]{Ref5}
V. V. Kozlov, V. P. Palamodov, \emph{On asymptotic solutions of the equations of classical mechanics},
Dokl. Akad. Nauk SSSR \textbf{263} 285--289; English transl., Soviet Math. Dokl., \textbf{25} (1982), 335--339.

\bibitem[Lag]{Lagrange}
J. L. Lagrange, \emph{M\`ecanique analytique}, Veuve Desaint, Paris (1788).

\bibitem[La]{Laloy1}
M. Laloy, \emph{On equilibrium instability for conservative and partially dissipative systems}, Internat. J. Non Linear Mech. {\bf 11} (1976) 295--301.


\bibitem[Ly]{Lyapunov}
A. M. Lyapunov, \emph{General problem of the stability of motion}, (Kharkov Math. Soc, Kharkov, 1892; French transl.) Ann. of Math. Studies, {\bf 17} (1947).

\bibitem[LP]{Laloy}
M. Laloy, K. Peiffer, \emph{On the instability of equilibrium when the potential function has a non-strict local minimum}, Arch. Rational Mech. Anal., \textbf{78} (1982), 213--22.


\bibitem[Mi]{Milnor}
J. W. Milnor, \emph{Topology from the differentiable viewpoint}, The University Press of Virginia, Charlottesville, U.S.A., (1965).


\bibitem[MN]{Ref6}
V. Moauro, P. Negrini, \emph{On the inversion of Lagrange\,-\,Dirichlet theorem}, Differential Integral Equations, \textbf{2} (1989), 471--8.

\bibitem[MT]{Trang_Massey}
D. B. Massey, L. D. Tr\'ang, \emph{Notes on Real and Complex Analytic and Semianalytic Singularities}, Unpublished notes.

\bibitem[Pa]{Palamodov}
V. P. Palamodov, \emph{Stability of motion and algebraic geometry}, Transl. Amer. Math. Soc., {\bf 168} (1995), 5--20.


\bibitem[Pa2]{Pa3}
V. P. Palamodov, \emph{Stability of equilibrium in a potential field}, Functional Anal. Appl., {\bf 11} (1977), 277--289.


\bibitem[So]{Sofer}
M. Sofer, \emph{On the inversion of the Lagrange\,-\,Dirichlet stability theorem - mechanical and generalized systems}, ZAMP {\bf 34} (1983), 1--12.

\bibitem[Sp]{Mark}
M. Spivakovsky, \emph{Resolution of Singularities: an Introduction}, Handbook of Geometry and Topology of Singularities I, Springer International Publishing, (2020), 183--242.

\bibitem[SS]{Soucek}
J. Souček, V. Souček, \emph{Morse-Sard theorem for real-analytic functions}, Commentationes Mathematicae Universitatis Carolinae, {\bf 13} (1972), No. 1, 45--51.

\bibitem[Ta]{Ta}
S. D. Taliaferro, \emph{Stability for two dimensional analytic potentials}, J. Differential Equations, {\bf 35} (1980), 248--265.

\bibitem[Ta2]{Taliaferro}
S. D. Taliaferro, \emph{An inversion of the Lagrange\,-\,Dirichlet stability theorem}, Arch. Rat. Mech. Anal., {\bf 73} (1980), 183--190.

\bibitem[Ta3]{Ref7}
S. D. Taliaferro, \emph{Instability of an equilibrium in a potential field}, Arch. Rational Mech. Anal., {\bf 109} (1990), 183--94.

\bibitem[Ur]{Urena}
A. J. Ureña, \emph{To what extent are unstable the maxima of the potential?}, Ann. Mat. Pura Appl., {\bf 199} (2020), 1763--1775.

\bibitem[Vi1]{Villamayor1}
O. Villamayor, \emph{Constructiveness of Hironaka's resolution}, Ann. Sci. Ecole Norm. Sup., Série 4, {\bf 22} (1989), 1--32.

\bibitem[Vi2]{Villamayor2}
O. Villamayor, \emph{Patching local uniformizations}, Ann. Sci. Ecole Norm. Sup., Série 4, {\bf 25} (1992), 629--677.

\bibitem[Wi]{Wi}
A. Wintner, \emph{The Analytical Foundations of Celestial Mechanics}, University Press. Princeton (1941).

\bibitem[W\l{}]{Wlodarczyk}
J. W\l{}odarczyk, \emph{Resolution of singularities of analytic spaces}, Proceedings of 15th Gökova Geometry-Topology Conference, 31--63.


\end{thebibliography}
\end{document}